\documentclass[11pt,a4paper,reqno]{amsart}
\usepackage[english]{babel}
\usepackage[T1]{fontenc}
\usepackage{verbatim}
\usepackage{mathabx}
\usepackage{palatino}
\usepackage{amsmath}
\usepackage{amssymb}
\usepackage{amsthm}
\usepackage{amsfonts}
\usepackage{graphicx}
\usepackage{color}
\usepackage{mathtools}
\usepackage{dsfont}
\usepackage{mathrsfs}
\usepackage{comment}

\DeclarePairedDelimiter\floor{\lfloor}{\rfloor}
\usepackage[colorlinks = true, citecolor = black]{hyperref}
\title[Plenty of big projections imply BPLG]{Plenty of big projections imply \\ big pieces of Lipschitz graphs}
\author{Tuomas Orponen}
\address{University of Jyv\"askyl\"a, Department of Mathematics and Statistics}
\email{tuomas.t.orponen@jyu.fi}
\date{\today}
\thanks{T.O. is supported by the Academy of Finland via the project \emph{Quantitative rectifiability in Euclidean and non-Euclidean spaces}, grant Nos. 309365, 314172, and via the project \emph{Incidences on Fractals}, grant No. 321896.}

\newcommand{\R}{\mathbb{R}}

\newcommand{\N}{\mathbb{N}}

\newcommand{\Z}{\mathbb{Z}}

\newcommand{\calT}{\mathcal{T}}
\newcommand{\calL}{\mathcal{L}}
\newcommand{\calD}{\mathcal{D}}
\newcommand{\calH}{\mathcal{H}}

\newcommand{\calB}{\mathcal{B}}
\newcommand{\calF}{\mathcal{F}}

\newcommand{\spt}{\operatorname{spt}}

\newcommand{\spa}{\operatorname{span}}
\newcommand{\diam}{\operatorname{diam}}
\newcommand{\card}{\operatorname{card}}
\newcommand{\dist}{\operatorname{dist}}

\newcommand{\w}{\mathrm{width}}
\newcommand{\ww}{\mathfrak{w}}

\def\Barint_#1{\mathchoice
          {\mathop{\vrule width 6pt height 3 pt depth -2.5pt
                  \kern -8pt \intop}\nolimits_{#1}}%
          {\mathop{\vrule width 5pt height 3 pt depth -2.6pt
                  \kern -6pt \intop}\nolimits_{#1}}%
          {\mathop{\vrule width 5pt height 3 pt depth -2.6pt
                  \kern -6pt \intop}\nolimits_{#1}}%
          {\mathop{\vrule width 5pt height 3 pt depth -2.6pt
                  \kern -6pt \intop}\nolimits_{#1}}}

\numberwithin{equation}{section}

\theoremstyle{plain}
\newtheorem{thm}[equation]{Theorem}

\newtheorem{lemma}[equation]{Lemma}

\newtheorem{proposition}[equation]{Proposition}
\newtheorem{question}{Question}
\newtheorem{claim}[equation]{Claim}

\theoremstyle{definition}

\newtheorem{definition}[equation]{Definition}

\theoremstyle{remark}
\newtheorem{remark}[equation]{Remark}

\newcommand{\nref}[1]{(\hyperref[#1]{#1})}

\begin{document}

\begin{abstract} I prove that closed $n$-regular sets $E \subset \R^{d}$ with plenty of big projections have big pieces of Lipschitz graphs. In particular, these sets are uniformly $n$-rectifiable. This answers a question of David and Semmes from 1993. \end{abstract}

\maketitle

\tableofcontents

\section{Introduction}

I start by introducing the key concepts of the paper. A Radon measure $\mu$ on $\R^{d}$ is called \emph{$s$-regular}, $s \geq 0$, if there exists a constant $C_{0} \geq 1$ such that
\begin{displaymath} C_{0}^{-1}r^{s} \leq \mu(B(x,r)) \leq C_{0}r^{s}, \qquad x \in \spt \mu, \, 0 < r < \diam(\spt \mu). \end{displaymath}
A set $E \subset \R^{d}$ is called $s$-regular if $E$ is closed, and the restriction of $s$-dimensional Hausdorff measure $\calH^{s}$ on $E$ is an $s$-regular Radon measure. An $n$-regular set $E \subset \R^{d}$ has \emph{big pieces of Lipschitz graphs} (BPLG) if the following holds for some constants $\theta,L > 0$: for every $x \in E$ and $0 < r < \diam(E)$, there exists an $n$-dimensional $L$-Lipschitz graph $\Gamma \subset \R^{d}$, which may depend on $x$ and $r$, such that
\begin{equation}\label{BPLG} \calH^{n}( B(x,r) \cap E \cap \Gamma) \geq \theta r^{n}. \end{equation}
By an $n$-dimensional $L$-Lipschitz graph, I mean a set of the form $\Gamma = \{v + f(v) : v \in V\}$, where $V \subset \R^{d}$ is an $n$-dimensional subspace, and $f \colon V \to V^{\perp}$ is $L$-Lipschitz. Sometimes it is convenient to call $\Gamma = \{v + f(v): v \in V\}$ an \emph{$L$-Lipschitz graph over $V$}. The BPLG property is stronger than \emph{uniform $n$-rectifiability}, see Section \ref{s:UR} for more discussion.

Let $G(d,n)$ be the \emph{Grassmannian} of all $n$-dimensional subspaces of $\R^{d}$, equipped with a natural metric which is invariant under the action of the orthogonal group $\mathcal{O}(d)$. See Section \ref{s:Grassmannian} for details. For $V \in G(d,n)$, let $\pi_{V}$ be the orthogonal projection to $V$. It is straightforward to check, see \cite[Proposition 1.4]{MR3790063}, that if $E \subset \R^{d}$ is an $n$-regular set with BPLG, then $E$ has many projections of positive $\calH^{n}$ measure: more accurately, if $\Gamma$ in \eqref{BPLG} is an $L$-Lipschitz graph over $V_{0} \in G(d,n)$, then there is a constant $\delta > 0$, depending only on $d,L,\theta$, such that
\begin{displaymath} \calH^{n}(\pi_{V}(B(x,r) \cap E)) \geq \calH^{n}(\pi_{V}(B(x,r) \cap E \cap \Gamma)) \geq \delta r^{n}, \qquad V \in B_{G(d,n)}(V_{0},\delta). \end{displaymath} 
David and Semmes asked in their 1993 paper \cite{MR1132876} whether a converse holds: are sets with BPLG precisely the ones with plenty of big projections? The problem is also mentioned in the monograph \cite[p. 29]{MR1251061} and, less precisely, in the 1994 ICM lecture of Semmes \cite{MR1403987}.
\begin{definition}[BP and PBP]\label{d:PBP} An $n$-regular set $E \subset \R^{d}$ has \emph{big projections} (BP) if there exists a constant $\delta > 0$ such that the following holds. For every $x \in E$ and $0 < r < \diam(E)$, there exists at least one plane $V = V_{x,r} \in G(d,n)$ such that 
\begin{equation}\label{form82} \calH^{n}(\pi_{V}(B(x,r) \cap E)) \geq \delta r^{n}. \end{equation}
The set $E$ has \emph{plenty of big projections} (PBP) if \eqref{form82} holds for all $V \in B(V_{x,r},\delta)$.
\end{definition}
In \cite[Definition 1.12]{MR1132876}, the PBP condition was called \emph{big projections in plenty of directions}. As noted above Definition \ref{d:PBP}, sets with BPLG have PBP. Conversely, one of the main results in \cite{MR1132876} states that even the weaker "single big projection" condition BP is sufficient to imply BPLG if it is paired with the following \emph{a priori} geometric hypothesis:
\begin{definition}[WGL]\label{wgl} An $n$-regular set $E \subset \R^{d}$ satisfies the \emph{weak geometric lemma} (WGL) if for all $\epsilon > 0$ there exists a constant $C(\epsilon) > 0$ such that the following (Carleson packing condition) holds:
\begin{displaymath} \int_{0}^{R} \calH^{n}(\{x \in E \cap B(x_{0},R) : \beta(B(x,r)) \geq \epsilon\}) \, \frac{dr}{r} \leq C(\epsilon)R, \quad x_{0} \in E, \, 0 < R < \diam(E). \end{displaymath} 
\end{definition}
In the definition above, the quantity $\beta(B(x,r))$ could mean a number of different things without changing the class of $n$-regular sets satisfying Definition \ref{wgl}. In the current paper, the most convenient choice is
\begin{displaymath} \beta(B(x,r)) := \beta_{1}(B(x,r)) := \inf_{V \in \mathcal{A}(d,n)} \frac{1}{r^{n}} \int_{B(x,r)} \frac{\dist(y,V)}{r} \, d\mu(y) \end{displaymath} 
with $\mu := \calH^{n}|_{E}$, and where $\mathcal{A}(d,n)$ is the "affine Grassmannian" of all $n$-dimensional planes in $\R^{d}$. The $\beta$-number above is an "$L^{1}$-variant" of the original "$L^{\infty}$-based $\beta$-number" introduced by Jones \cite{Jones:TSP}, namely
\begin{displaymath} \beta_{\infty}(B(x,r)) := \inf_{V \in \mathcal{A}(d,n)} \sup_{y \in E \cap B(x,r)} \frac{\dist(y,V)}{r}. \end{displaymath}
If $E \subset \R^{d}$ is $n$-regular, then the following relation holds between the two $\beta$-numbers:
\begin{displaymath} \beta_{\infty}(B(x,r)) \lesssim \beta(B(x,2r))^{1/(n + 1)}, \qquad x \in E, \, 0 < r < \diam(E). \end{displaymath} 
For a proof, see \cite[p. 28]{DS1}. This inequality shows that the WGL, a condition concerning all $\epsilon > 0$ simultaneously, holds for the numbers $\beta(B(x,r))$ if and only if it holds for the numbers $\beta_{\infty}(B(x,r))$.

After these preliminaries, the result of David and Semmes \cite[Theorem 1.14]{MR1132876} can be stated as follows:
\begin{thm}[David-Semmes]\label{t:DS} An $n$-regular set $E \subset \R^{d}$ has \textup{BPLG} if and only if $E$ has \textup{BP} and satisfies the \textup{WGL}. 
\end{thm}
The four corners Cantor set has BP (find a direction where the projections of the four boxes tile an interval) but fails to have BPLG, being purely $1$-unrectifiable. This means that the WGL hypothesis cannot be omitted from the previous statement. However, the four corners Cantor set fails to have PBP, by the Besicovitch projection theorem \cite{MR1513231}, which states that almost every projection of a purely $1$-rectifiable set of $\sigma$-finite length has measure zero. The main result of this paper shows that PBP alone implies BPLG:
\begin{thm}\label{main} Let $E \subset \R^{d}$ be an $n$-regular set with \textup{PBP}. Then $E$ has \textup{BPLG}.  \end{thm}
To prove Theorem \ref{main}, all one needs to show is that
\begin{displaymath} \textup{PBP} \quad \Longrightarrow \quad \textup{WGL}. \end{displaymath}
The rest then follows from the work of David and Semmes, Theorem \ref{t:DS}. 

\subsection{Connection to uniform rectifiability}\label{s:UR} The BPLG property is a close relative of \emph{uniform $n$-rectifiability}, introduced by David and Semmes \cite{DS1} in the early 90s. An $n$-regular set $E \subset \R^{d}$ is uniformly $n$-rectifiable, $n$-UR in brief, if \eqref{BPLG} holds for some \emph{$n$-dimensional $L$-Lipschitz images} $\Gamma = f(B(0,r))$, with $B(0,r) \subset \R^{n}$, instead of $n$-dimensional $L$-Lipschitz graphs. As shown by David and Semmes in \cite{DS1,MR1251061}, the $n$-UR property has many equivalent, often surprising characterisations: for example, singular integrals with odd $n$-dimensional kernels are $L^{2}$-bounded on an $n$-regular set $E \subset \R^{d}$ if and only if $E$ is $n$-UR. Since its conception, the study of uniform (and, more generally, quantitative) rectifiability has become an increasingly popular topic, for a good reason: techniques in the area have proven fruitful in solving long-standing problems on harmonic measure and elliptic PDEs \cite{MR3540451,MR3707490,MR3807315,ntov}, theoretical computer science \cite{NY}, and metric embedding theory \cite{NY2}. This list of references is hopelessly incomplete!

Since $n$-dimensional Lipschitz graphs can be written as $n$-dimensional Lipschitz images, $n$-regular sets with BPLG are $n$-UR. In particular, Theorem \ref{main} implies that $n$-regular sets with PBP are $n$-UR. The converse is false: Hrycak (unpublished) observed in the 90s that a simple iterative construction can be used to produce $1$-regular compact sets $K_{\epsilon} \subset \R^{2}$, $\epsilon > 0$, with the properties
\begin{itemize}
\item[(a)] $\mathcal{H}^{1}(K_{\epsilon}) = 1$ and $\calH^{1}(\pi_{L}(K_{\epsilon})) < \epsilon$ for all $L \in G(2,1)$, 
\item[(b)] $K_{\epsilon}$ is $1$-UR with constants independent of $\epsilon > 0$.
\end{itemize}
This means that UR sets do not necessarily have PBP, or at least bounds for $n$-UR constants do not imply bounds for PBP constants. The details of Hrycak's construction are contained in the appendix of Azzam's paper \cite{2017arXiv171103088A}, but they can also be outlined in a few words: pick $n := \floor{\epsilon^{-1}}$. Sub-divide $I_{0} := [0,1] \times \{0\} \subset \R^{2}$ into $n$ segments $I_{1},\ldots,I_{n}$ of equal length, and rotate them individually counter-clockwise by $2\pi/n$. Then, sub-divide each $I_{j}$ into $n$ segments of equal length, and rotate by $2\pi/n$ again. Repeat this procedure $n$ times to obtain a compact set $K_{n} = K_{\epsilon}$ consisting of $n^{n}$ segments of length $n^{-n}$. It is not hard to check that (a) and (b) hold for $K_{\epsilon}$. In particular, to check (b), one can easily cover $K_{\epsilon}$ by a single $1$-regular continuum $\Gamma \subset B(0,2)$ of length $\calH^{1}(\Gamma) \leq 10$.

\subsection{Previous and related work} It follows from the Besicovitch-Federer projection theorem \cite{MR1513231,MR22594} that an $n$-regular set with PBP is $n$-rectifiable. The challenge in proving Theorem \ref{main} is to upgrade this "qualitative" property to BPLG. For general compact sets in $\R^{2}$ of finite $1$-dimensional measure, a quantitative version of the Besicovitch projection theorem is due to Tao \cite{MR2500864}. It appears, however, that Theorem \ref{main} does not follow from his work, not even in $\R^{2}$. Another, more recent, result for general $n$-regular sets is due to Martikainen and myself \cite{MR3790063}: the main result of \cite{MR3790063} shows that BPLG is equivalent to a property (superficially) stronger than PBP. This property roughly states that the $\pi_{V}$-projections of the measure $\calH^{n}|_{E}$ lie in $L^{2}(V)$ on average over $V \in B_{G(d,n)}(V_{0},\delta)$. One of the main propositions from \cite{MR3790063} also plays a part in the present paper, see Proposition \ref{MOProp}. Interestingly, while the main result of the current paper is formally stronger than the result in \cite{MR3790063}, the new proof does not supersede the previous one: in \cite{MR3790063}, the $L^{2}$-type assumption in a fixed ball was used to produce a big piece of a Lipschitz graph in the very same ball. Here, on the contrary, PBP needs to be employed in many balls, potentially much smaller than the "fixed ball" one is interested in. Whether this is necessary or not is posed as Question \ref{q1} below.

Besides Tao's paper mentioned above, there is plenty of recent activity around the problem of quantifying Besicovitch's projection theorem, that is, showing that "quantitatively unrectifiable sets" have quantifiably small projections. As far as I know, Tao's paper is the only one dealing with general sets, while other authors, including Bateman, Bond, {\L}aba, Nazarov, Peres, Solomyak, and Volberg have concentrated on self-similar sets of various generality \cite{MR2727621,MR3188064,MR2652491,MR3129103,MR3526481,MR2641082,MR1907902}. In these works, strong upper (and some surprising lower) bounds are obtained for the \emph{Favard length} of the $k^{th}$ iterate of self-similar sets. In the most recent development \cite{2020arXiv200303620C}, Cladek, Davey, and Taylor considered the \emph{Favard curve length} of the four corners Cantor set.

Quantifying the Besicovitch projection theorem is related to an old problem of Vitushkin. The remaining open question is to determine whether arbitrary compact sets $E \subset \R^{2}$ of positive Favard length have positive \emph{analytic capacity}. It seems unlikely that the method of the present paper would have any bearing on Vitushkin's problem, but the questions are not entirely unrelated either: I refer to the excellent introduction in the paper \cite{2017arXiv171200594C} of Chang and Tolsa for more details.

Finally, Theorem \ref{main} can be simply viewed as a characterisation of the BPLG property, of which there are not many available -- in contrast to uniform rectifiability, which is charaterised by seven conditions in \cite{DS1} alone! I already mentioned that BPLG is equivalent to BP+WGL by \cite{MR1132876}, and that with Martikainen \cite{MR3790063}, we characterised BPLG via the $L^{2}$-norms of the projections $\pi_{V\sharp}\mathcal{H}^{n}|_{E}$. Another, very recent, characterisation of BPLG, in terms of \emph{conical energies}, is due to {D{\k{a}}browski} \cite{2020arXiv200614432D}.

\subsection{An open problem} An answer to the question below does not seem to follow from the method of this paper. 
\begin{question}\label{q1} For all $\delta > 0$ and $C_{0} \geq 1$, do there exist $L \geq 1$ and $\theta > 0$ such that the following holds? Whenever $E \subset \R^{d}$ is an $n$-regular set with regularity constant at most $C_{0}$, and
\begin{equation}\label{form111} \calH^{n}(\pi_{V}(B(0,1) \cap E)) \geq \delta, \qquad V \in B_{G(d,n)}(V_{0},\delta), \end{equation}
then there exists an $n$-dimensional $L$-Lipschitz graph $\Gamma \subset \R^{d}$ such that $\calH^{n}(E \cap \Gamma) \geq \theta$. \end{question} 
In addition to the "single scale" assumption \eqref{form111}, the proof of Theorem \ref{main} requires information about balls much smaller than $B(0,1)$ to produce the Lipschitz graph $\Gamma$.

\subsection{Notation} An open ball in $\R^{d}$ with centre $x \in \R^{d}$ and radius $r > 0$ will be denoted $B(x,r)$. When $x = 0$, I sometimes abbreviate $B(x,r) =: B(r)$. The notations $\mathrm{rad}(B)$ and $\diam(B)$ mean the radius and diameter of a ball $B \subset \R^{d}$, respectively, and $\lambda B := B(x,\lambda r)$ for $B = B(x,r)$ and $\lambda > 0$.

For $A,B > 0$, the notation $A \lesssim_{p_{1},\ldots,p_{k}} B$ means that there exists a constant $C \geq 1$, depending only on the parameters $p_{1},\ldots,p_{k}$, such that $A \leq CB$. Very often, one of these parameters is either the ambient dimension "$d$", or then the PBP or $n$-regularity constant "$\delta$" or "$C_{0}$" of a fixed $n$-regular set $E \subset \R^{d}$ having PBP, that is, satisfying the hypotheses of Theorem \ref{main}. In these cases, the dependence is typically omitted from the notation: in other words, $A \lesssim_{d,\delta,C_{0}} B$ is abbreviated to $A \lesssim B$. The two-sided inequality $A \lesssim_{p} B \lesssim_{p} A$ is abbreviated to $A \sim_{p} B$, and $A \gtrsim_{p} B$ means the same as $B \lesssim_{p} A$.

\subsection{Acknowledgements} I would like to thank Michele Villa for useful conversations, and Alan Chang for pointing out a mistake in the proof of Lemma \ref{lemma1} in an earlier version of the paper. I'm also grateful to Damian {D{\k{a}}browski} for reading the paper carefully and giving many useful comments. Finally, I am grateful to the anonymous reviewers for their careful reading, and for spotting a large number of small inaccuracies.

\section{Preliminaries on the Grassmannian}\label{s:Grassmannian}

Before getting started, we gather here a few facts of the Grassmannian $G(d,n)$ of $n$-dimensional subspaces of $\R^{d}$. Here $0 \leq n \leq d$, and the extreme cases are $G(d,0) = \{0\}$ and $G(d,d) = \{\R^{d}\}$. We equip $G(d,n)$ with the metric
\begin{displaymath} d(V_{1},V_{2}) := \|\pi_{V_{1}} - \pi_{V_{2}}\|, \qquad V_{1},V_{2} \in G(d,n), \end{displaymath}
where $\|\cdot\|$ refers to operator norm. That "$d$" means two different things here is regrettable, but the correct interpretation should always be clear from context, and the metric "$d$" will only be used very occasionally. The metric space $(G(d,n),d)$ is compact, and open balls in $G(d,n)$ will be denoted $B_{G(d,n)}(V,r)$. An equivalent metric on $G(d,n)$ is given by
\begin{displaymath} \bar{d}(V_{1},V_{2}) := \max\{\dist(v_{1},V_{2}) : v_{1} \in V_{1} \text{ and } |v_{1}| = 1\}. \end{displaymath}
For a proof, see \cite[Lemma 4.1]{MR2729011}. With the equivalence of $d$ and $\bar{d}$ in hand, we easily infer the following auxiliary result:
\begin{lemma}\label{GrLemma} Let $0 < n < d$, and let $W_{1},W_{2} \in G(d,n + 1)$, and let $V_{1} \in G(d,n)$ with $V_{1} \subset W_{1}$. Then, there exists $V_{2} \in G(d,n)$ such that $V_{2} \subset W_{2}$ and $d(V_{1},V_{2}) \lesssim d(W_{1},W_{2})$.
\end{lemma}

\begin{proof} By the equivalence of $d$ and $\bar{d}$, we have $r := \bar{d}(W_{1},W_{2}) \lesssim d(W_{1},W_{2})$. We may assume that $r$ is small, depending on the ambient dimension, otherwise any $n$-dimensional subspace $V_{2} \subset W_{2}$ satisfies $d(V_{1},V_{2}) \leq \diam G(d,n) \lesssim r$. Now, let $\{e_{1},\ldots,e_{n}\}$ be an orthonormal basis for $V_{1}$, and for all $e_{j} \in V_{1} \subset W_{1}$, pick some $\bar{e}_{j} \in W_{2}$ with $|e_{j} - \bar{e}_{j}| \leq r$. If $r > 0$ is small enough, the vectors $\bar{e}_{1},\ldots,\bar{e}_{n}$ are linearly independent, hence span an $n$-dimensional subspace $V_{2} \subset W_{2}$. Since $|e_{j} - \bar{e}_{j}| \leq r$ for all $1 \leq j \leq n$, an arbitrary unit vector $v_{1} = \sum \beta_{j}e_{j} \in V_{1}$ lies at distance $\lesssim r$ from $v_{2} := \sum \beta_{j}\bar{e}_{j} \in V_{2}$, and consequently
\begin{displaymath} d(V_{1},V_{2}) \sim \bar{d}(V_{1},V_{2}) = \max\{\dist(v_{1},V_{2}) : v_{1} \in V_{1} \text{ and } |v_{1}| = 1\} \lesssim r. \end{displaymath} 
This completes the proof. \end{proof}
We will often use the standard "Haar" probability measure $\gamma_{d,n}$ on $G(d,n)$. Namely, let $\theta_{d}$ be the Haar measure on the orthogonal group $\mathcal{O}(d)$, and define
\begin{displaymath} \gamma_{d,n}(\mathcal{V}) := \theta_{d}(\{g \in \mathcal{O}(d) : gV_{0} \in \mathcal{V}\}), \qquad \mathcal{V} \subset G(d,n), \end{displaymath} 
where $V_{0} \in G(d,n)$ is any fixed subspace. The measure $\gamma_{d,n}$ is the unique $\mathcal{O}(d)$-invariant Radon probability measure on $G(d,n)$, see \cite[\S 3.9]{zbMATH01249699}. At a fairly late stage of the proof of Theorem \ref{main}, we will need the following "Fubini" theorem for the measure $G(d,n)$:
\begin{lemma}\label{l:fubini} Let $0 < n < d$. For $W \in G(d,n + 1)$, let $G(W,n) := \{V \in G(d,n) : V \subset W\}$. Then $G(W,n)$ can be identified with $G(n + 1,n)$, and we equip $G(W,n)$ with the Haar measure $\gamma_{W,n + 1,n} := \gamma_{n + 1,n}$, constructed as above. Then, the following holds for all Borel sets $B \subset G(d,n)$:
\begin{equation}\label{form109} \gamma_{d,n}(B) = \int_{G(d,n + 1)} \gamma_{W,n + 1,n}(B) \, d\gamma_{d,n + 1}(W). \end{equation}
\end{lemma}

\begin{proof} This is the same argument as in \cite[Lemma 3.13]{zbMATH01249699}: one simply checks that both sides of \eqref{form109} define $\mathcal{O}(d)$-invariant probability measures on $\gamma_{d,n}$, and then appeals to the uniqueness of such measures. \end{proof}

We record one final auxiliary result: 

\begin{lemma}\label{MOLemma} For all $0 < n < d$, $\delta > 0$, there exists an "angle" $\alpha = \alpha(d,\delta) > 0$ such that the following holds. If $z \in \R^{d}$, and $V \in G(d,n)$ satisfy $|\pi_{V}(z)| \leq \alpha |z|$, then there exists a plane $V' \in G(d,n)$ with $d(V,V') < \delta$ such that $\pi_{V'}(z) = 0$. \end{lemma} 

\begin{proof} The proof of \cite[Lemma A.1]{MR3790063} begins by establishing exactly this claim, although the statement of \cite[Lemma A.1]{MR3790063} does not mention it explicitly.  \end{proof}

\section{Dyadic reformulations}

\subsection{Dyadic cubes} It is known (see for example \cite[\S 2]{MR1132876}) that an $n$-regular set $E \subset \R^{d}$ supports a system $\mathcal{D}$ of "dyadic cubes", that is, a collection of subset of $E$ with the following properties. First, $\calD$ can be written as a disjoint union
\begin{displaymath} \calD = \bigcup_{j \in \Z} \calD_{j}, \end{displaymath}
where the elements $Q \in \calD_{j}$ are referred to as \emph{cubes of side-length $2^{-j}$}. For $j \in \Z$ fixed, the sets of $\calD_{j}$ are disjoint and cover $E$. For $Q \in \calD_{j}$, one writes $\ell(Q) := 2^{-j}$. The side-length $\ell(Q)$ is related to the geometry of $Q \in \calD_{j}$ in the following way: there are constants $0 < c < C < \infty$, and points $c_{Q} \in Q \subset E$ (known as the "centres" of $Q \in \calD$) with the properties
\begin{displaymath} B(c_{Q},c\ell(Q)) \cap E \subset Q \subset B(c_{Q},C\ell(Q)). \end{displaymath}
In particular, it follows from the $n$-regularity of $E$ that $\mu(Q) \sim \ell(Q)^{n}$ for all $Q \in \calD$. The balls $B(c_{Q},C\ell(Q))$ containing $Q$ are so useful that they will have an abbreviation:
\begin{displaymath} B_{Q} := B(c_{Q},C\ell(Q)). \end{displaymath}
If we choose the constant $C \geq 1$ is large enough, as we do, the balls $B_{Q}$ have the property
\begin{displaymath} Q \subset Q' \quad \Longrightarrow \quad B_{Q} \subset B_{Q'}. \end{displaymath}
The "dyadic" structure of the cubes in $\calD$ is encapsulated by the following properties: 
\begin{itemize}
\item For all $Q,Q' \in \calD$, either $Q \subset Q'$, or $Q' \subset Q$, or $Q \cap Q' = \emptyset$. 
\item Every $Q \in \calD_{j}$ has as \emph{parent} $\hat{Q} \in \calD_{j - 1}$ with $Q \subset \hat{Q}$.
\end{itemize}
If $Q \in \calD_{j}$, the cubes in $\calD_{j + 1}$ whose parent is $Q$ are known as the \emph{children of $Q$}, denoted $\mathbf{ch}(Q)$. The \emph{ancestry} of $Q$ consists of all the cubes in $\calD$ containing $Q$.

A small technicality arises if $\diam(E) < \infty$: then the collections $\calD_{j}$ are declared empty for all $j < j_{0}$, and $\calD_{j_{0}}$ contains a unique element, known as the \emph{top cube} of $\calD$. All of the statements above hold in this scenario, except that the top cube has no parents.  

\subsection{Dyadic reformulations of PBP and WGL} Let us next reformulate some of the conditions familiar from the introduction in terms of a fixed dyadic system $\calD$ on $E$. 
\begin{definition}[PBP] An $n$-regular set $E \subset \R^{d}$ has PBP if there exists $\delta > 0$ such that the following holds. For all $Q \in \calD$, there exists a ball $S_{Q} \subset G(d,n)$ of radius $\mathrm{rad}(S_{Q}) \geq \delta$ such that
\begin{displaymath} \calH^{n}(\pi_{V}(E \cap B_{Q})) \geq \delta \mu(Q), \qquad V \in S_{Q}. \end{displaymath}
\end{definition}
It is easy to see that the dyadic PBP is equivalent to the continuous PBP: in particular, the dyadic PBP follows by applying the continuous PBP to the ball $B_{Q} = B(c_{Q},C\ell(Q))$ centred at $c_{Q} \in E$. Only the dyadic PBP will be used below.
\begin{definition}[WGL] An $n$-regular set $E \subset \R^{d}$ satisfies the WGL if for all $\epsilon > 0$, there exists a constant $C(\epsilon) > 0$ such that the following holds:
\begin{displaymath} \mathop{\sum_{Q \in \calD(Q_{0})}}_{\beta(Q) \geq \epsilon} \mu(Q) \leq C(\epsilon)\mu(Q_{0}), \qquad Q_{0} \in \calD. \end{displaymath}
Here $\mu := \calH^{n}|_{E}$, $\beta(Q) := \beta(B_{Q})$, and $\calD(Q_{0}) := \{Q \in \calD : Q \subset Q_{0}\}$. \end{definition}
It is well-known, but takes a little more work to show, that the dyadic WGL is equivalent to the continuous WGL; this fact is stated without proof in numerous references, for example \cite[(2.17)]{MR1132876}. I also leave the checking to the reader. 

One often wishes to decompose $\calD$, or subsets thereof, into \emph{trees}:
\begin{definition}[Trees] Let $E \subset \R^{d}$ be an $n$-regular set with associated dyadic system $\calD$. A collection $\mathcal{T} \subset \mathcal{D}$ is called a \emph{tree} if the following conditions are met:
\begin{itemize}
\item $\mathcal{T}$ has a \emph{top cube} $Q(\mathcal{T}) \in \mathcal{T}$ with the property that $Q \subset Q(\calT)$ for all $Q \in \mathcal{T}$.
\item $\calT$ is \emph{consistent}: if $Q_{1},Q_{3} \in \calT$, $Q_{2} \in \mathcal{D}$, and $Q_{1} \subset Q_{2} \subset Q_{3}$, then $Q_{2} \in \mathcal{T}$.
\item If $Q \in \mathcal{T}$, then either $\mathbf{ch}(Q) \subset \calT$ or $\mathbf{ch}(Q) \cap \mathcal{T} = \emptyset$.
\end{itemize}
The final axiom allows to define the \emph{leaves} of $\calT$ consistently: these are the cubes $Q \in \mathcal{T}$ such that $\mathbf{ch}(Q) \cap \mathcal{T} = \emptyset$. The leaves of $\calT$ are denoted $\mathbf{Leaves}(\mathcal{T})$. The collection $\mathbf{Leaves}(\mathcal{T})$ always consists of disjoint cubes, and it may happen that $\mathbf{Leaves}(\mathcal{T}) = \emptyset$.
\end{definition}

Some trees will be used to prove the following reformulation of the WGL:

\begin{lemma}\label{l:WGL} Let $E \subset \R^{d}$ be an $n$-regular set supporting a collection $\mathcal{D}$ of dyadic cubes. Let $\mu := \calH^{n}|_{E}$. Assume that for all $\epsilon > 0$, there exists $N = N(\epsilon) \in \N$ such that the following holds:
\begin{equation}\label{form80} \mu(\{x \in Q : \card\{Q' \in \calD : x \in Q' \subset Q \text{ and } \beta(Q') \geq \epsilon\} \geq N\}) \leq \tfrac{1}{2}\mu(Q), \qquad Q \in \calD. \end{equation}
Then $E$ satisfies the WGL.
\end{lemma}

\begin{remark} Chebyshev's inequality applied to the set $\{x \in Q : \sum_{Q' \subset Q, \beta(Q') \geq \epsilon} \mathbf{1}_{Q'}(x) \geq N\}$ shows that the WGL implies \eqref{form80}. Therefore \eqref{form80} is equivalent to the WGL. \end{remark}

\begin{proof}[Proof of Lemma \ref{l:WGL}] Fix $Q_{0} \in \mathcal{D}$ and $\epsilon > 0$. We will show that
\begin{equation}\label{form81} \mathop{\sum_{Q \subset Q_{0}}}_{\beta(Q) \geq \epsilon} \mu(Q) \leq 2N\mu(Q_{0}). \end{equation}
Abbreviate $\calD := \{Q \in \calD : Q \subset Q_{0}\}$, and decompose $\calD$ into trees by the following simple stopping rule. The first tree $\calT_{0}$ has top $Q(\calT_{0}) = Q_{0}$, and its leaves are the maximal cubes $Q \in \calD$ (if any should exist) such that
\begin{displaymath} \card\{Q' \in \calD : Q \subset Q' \subset Q_{0} \text{ and } \beta(Q') \geq \epsilon\} = N. \end{displaymath}
Here $N = N(\epsilon) \geq 1$, as in \eqref{form80}. All the children of previous generation leaves are declared to be new top cubes, under which new trees are constructed by the same stopping condition. Let $\calT_{0},\calT_{1},\ldots$ be the trees obtained by this process, with top cubes $Q_{0},Q_{1},\ldots$ Note that $\calD = \bigcup_{j \geq 0} \calT_{j}$, and
\begin{displaymath} \card\{Q \in \calT_{j} : x \in Q \text{ and } \beta(Q) \geq \epsilon\} \leq N, \qquad x \in Q_{j}. \end{displaymath} Further, \eqref{form80} implies that
\begin{displaymath} \mu(\cup \mathbf{Leaves}(\calT_{j})) \leq \tfrac{1}{2} \mu(Q_{j}), \qquad j \geq 0. \end{displaymath} 
On the other hand, the sets $E_{j} := Q_{j} \, \setminus \, \cup \mathbf{Leaves}(\calT_{j})$ are disjoint. Now, we may estimate as follows:
\begin{align*} \mathop{\sum_{Q \subset Q_{0}}}_{\beta(Q) \geq \epsilon} \mu(Q) = \sum_{j = 0}^{\infty} \int_{Q_{j}} \mathop{\sum_{Q \in \calT_{j}}}_{\beta(Q) \geq \epsilon} \mathbf{1}_{Q}(x) \, dx \leq N \sum_{j = 0}^{\infty} \mu(Q_{j}) \leq 2N \sum_{j = 0}^{\infty} \mu(E_{j}) \leq 2N\mu(Q_{0}). \end{align*} 
This completes the proof of \eqref{form81}.  \end{proof}

By Theorem \ref{t:DS}, the PBP condition together with the WGL implies BPLG, and the condition in Lemma \ref{l:WGL} is a reformulation of the WGL. Therefore, our main result, Theorem \ref{main}, will be a consequence of the next proposition:

\begin{proposition}\label{mainProp} Assume that $E \subset \R^{d}$ is an $n$-regular set with \textup{PBP}. Then, for every $\epsilon > 0$, there exists $N \geq 1$, depending on $d$, $\epsilon$, and the $n$-regularity and \textup{PBP} constants of $E$, such that the following holds. The sets
\begin{displaymath} E_{Q} := E_{Q}(N,\epsilon) := \left\{x \in Q : \card \{Q' \in \calD : x \in Q' \subset Q \text{ and } \beta(Q') \geq \epsilon\} \geq N \right\} \end{displaymath}
satisfy $\mu(E_{Q}) \leq \tfrac{1}{2}\mu(Q)$ for all $Q \in \mathcal{D}$.
\end{proposition}
Proving this proposition will occupy the rest of the paper.

\section{Construction of heavy trees}\label{s:trees}

The proof of Proposition \ref{mainProp} proceeds by counter assumption: there exists a cube $Q_{0} \in \mathcal{D}$, a small number $\epsilon > 0$, and a large number $N \geq 1$ of the form $N = KM$, where also $K,M \geq 1$ are large numbers, with the property 
\begin{equation}\label{form83} \mu(E_{Q_{0}}) \geq \tfrac{1}{2}\mu(Q_{0}). \end{equation}
This will lead to a contradiction if both $K$ and $M$ are large enough, depending on $d$, $\epsilon$, and the $n$-regularity and PBP constants of $E$. Precisely, $M \geq 1$ gets chosen first within the proof of Proposition \ref{treeProp}. The parameter $K \geq 1$ is chosen second, and depends also on $M$. For the details, see the proof of Proposition \ref{mainProp}, which can be found around \eqref{form55}.

From now on, we will restrict attention to sub-cubes of $Q_{0}$, and we abbreviate $\mathcal{D} := \mathcal{D}(Q_{0})$. We begin by using \eqref{form83}, and the definition of $E_{Q_{0}}$, to construct a number of \emph{heavy trees} $\calT_{0},\calT_{1},\ldots \subset \mathcal{D}$ with the following properties:
\begin{enumerate}
\item[(T1) \phantomsection\label{T1}] $\mu(E_{Q_{0}} \cap Q(\calT_{j})) \geq \tfrac{1}{4}\mu(Q(\calT_{j}))$ for all $j \geq 0$.
\item[(T2) \phantomsection\label{T2}] $E_{Q_{0}} \cap Q(\calT_{j}) \subset \cup \mathbf{Leaves}(\calT_{j})$ for all $j \geq 0$.
\item[(T3) \phantomsection\label{T3}] For every $j \geq 0$ and $Q \in \mathbf{Leaves}(\calT_{j})$ it holds
\begin{displaymath} \card\{Q' \in \calT_{j} : Q \subset Q' \subset Q(\calT_{j}) \text{ and } \beta(Q') \geq \epsilon\} = M. \end{displaymath}
\item[(T4) \phantomsection\label{T4}] The top cubes satisfy $\sum_{j} \mu(Q(\calT_{j})) \geq \tfrac{K}{4} \mu(Q_{0})$.
\end{enumerate}

Before constructing the trees with properties \nref{T1}-\nref{T4}, let us use them, combined with some auxiliary results, to complete the proof of Proposition \ref{mainProp}. The first ingredient is the following proposition:

\begin{proposition}\label{treeProp} If the parameter $M \geq 1$ is large enough, depending only on $d$, $\epsilon$, and the $n$-regularity and \textup{PBP} constants of $E$, then $\w(\calT_{j}) \geq \tau \mu(Q(\calT_{j}))$, where $\tau > 0$ depends only on $d$, and the $n$-regularity and \textup{PBP} constants of $E$.
\end{proposition}
Here $\w(\calT_{j}) = \sum_{Q \in \calT_{j}} \w(Q)\mu(Q)$ is a quantity to be properly introduced in Section \ref{s:width}. For now, we only need to know that the coefficients $\w(Q)$ satisfy a Carleson packing condition, depending only on the $n$-regularity constant of $E$:
\begin{displaymath} \w(\calD) := \sum_{Q \subset Q_{0}} \w(Q)\mu(Q) \lesssim \mu(Q_{0}). \end{displaymath}

We may then prove Proposition \ref{mainProp}:

\begin{proof}[Proof of Proposition \ref{mainProp}] Let $N = KM$, where $M \geq 1$ is chosen so large that the hypothesis of Proposition \ref{treeProp} is met: every heavy tree $\calT_{j}$ satisfies $\w(\calT_{j}) \geq \tau \mu(Q(\calT_{j}))$. According to \nref{T4} in the construction of the heavy trees, this implies
\begin{equation}\label{form55} \w(\calD) \geq \sum_{j \geq 0} \w(\calT_{j}) \geq \tau \sum_{j \geq 0} \mu(Q(\calT_{j})) \geq \frac{\tau K}{4}\mu(Q_{0}). \end{equation} 
Now, the lower bound in \eqref{form55} violates the Carleson packing condition for $\w(\calD)$ if the constant $K \geq 1$ is chosen large enough, depending on the admissible parameters. The proof of Proposition \ref{mainProp} is complete. \end{proof}

The rest of this section is spent constructing the heavy trees. We first construct a somewhat larger collection, and then prune it. In fact, the construction of the larger collection is already familiar from the proof of Lemma \ref{l:WGL}, with notational changes: the first tree $\calT_{0}$ has top $Q(\calT_{0}) = Q_{0}$, and its leaves consist of the maximal cubes $Q \in \calD$ with the property that
\begin{equation}\label{form47} \card \{Q' \in \mathcal{D} : Q \subset Q' \subset Q(\calT_{0}) \text{ and } \beta(Q') \geq \epsilon\} = M. \end{equation}
The tree $\calT_{0}$ itself consists of the cubes in $\calD$ which are not strict sub-cubes of some $Q \in \mathbf{Leaves}(\calT_{0})$. It is easy to check that $\calT_{0}$ is a tree. 

Assume then that some trees $\calT_{0},\ldots,\calT_{k}$ have already been constructed. Let $0 \leq j \leq k$ be an index such that for some $Q \in \mathbf{Leaves}(\mathcal{T}_{j})$, at least one cube $Q_{k + 1} \in \mathbf{ch}(Q)$ has not yet been assigned to any tree. The cube $Q_{k + 1}$ then becomes the top cube of a new tree $\mathcal{T}_{k + 1}$, thus $Q_{k + 1} = Q(\mathcal{T}_{k + 1})$. The tree $\mathcal{T}_{k + 1}$ is constructed with the same stopping condition \eqref{form47}, just replacing $Q(\mathcal{T}_{0})$ by $Q_{k + 1} = Q(\mathcal{T}_{k + 1})$. 

Note that if $\mathbf{Leaves}(\calT_{j}) = \emptyset$ for some $j \in \N$, then no further trees will be constructed with top cubes contained in $Q(\calT_{j})$. As a corollary of the stopping condition, we record the uniform upper bound
\begin{equation}\label{form51} \card\{Q \in \calT_{j} : x \in Q \text{ and } \beta(Q) \geq \epsilon\} \leq M, \qquad x \in Q(\calT_{j}), \, j \geq 0. \end{equation} 
We next prune the collection of trees. Let $\mathbf{Top}$ be the collection of all the top cubes $Q(\calT_{j})$ constructed above, and let $\mathbf{Top}_{K} \subset \mathbf{Top}$ be the maximal cubes with the property 
\begin{displaymath} \card\{Q' \in \mathbf{Top} : Q \subset Q' \subset Q_{0}\} = K. \end{displaymath} 
We discard all the trees whose tops are strictly contained in one of the cubes in $\mathbf{Top}_{K}$, and we re-index the remaining trees as $\calT_{0},\calT_{1},\calT_{2},\ldots$ Thus, the remaining trees are the ones whose top cube contains some element of $\mathbf{Top}_{K}$. We record that
\begin{equation}\label{form49} \card\{j \geq 0 : x \in Q(\calT_{j})\} \leq K, \qquad x \in Q_{0}. \end{equation} 
We write $\calT := \cup \calT_{j}$ for brevity. We claim that
\begin{equation}\label{form50} \card\{Q \in \calT : x \in Q \text{ and } \beta(Q) \geq \epsilon\} = N, \qquad x \in E_{Q_{0}}. \end{equation} 
Indeed, fix $x \in E_{Q_{0}}$, and recall that
\begin{equation}\label{form53} \card\{Q \in \calD : x \in Q \text{ and } \beta(Q) \geq \epsilon\} \geq N \end{equation}
by definition. We first claim that $x$ is contained in $\geq K + 1$ cubes in $\mathbf{Top}$. If $x$ was contained in $\leq K$ cubes in $\mathbf{Top}$, then $x$ would be contained in $\leq K - 1$ distinct leaves, and the stopping condition \eqref{form47} would imply that
\begin{equation}\label{form54} \card\{Q \in \calD : x \in Q \text{ and } \beta(Q) \geq \epsilon\} < (K - 1)M + M = N, \end{equation}
contradicting $x \in E_{Q_{0}}$. Therefore, $x$ is indeed contained in $K + 1$ cubes in $\mathbf{Top}$. Let the largest such top cubes be $Q_{0} \supset Q_{1} \supset \ldots \supset Q_{K - 1} \supset Q_{K}$, so $Q_{K - 1} \in \mathbf{Top}_{K}$. Now, it suffices to note that whenever $x \in Q_{j}$, $1 \leq j \leq K$, then $x$ is contained in some element of $\mathbf{Leaves}(\calT_{j - 1})$, which implies by the stopping condition that
\begin{equation}\label{form52} \card \{Q \in \calT_{j - 1} : x \in Q \text{ and } \beta(Q) \geq \epsilon\} = M. \end{equation}
Since $\calT_{j - 1} \subset \calT$ for $1 \leq j \leq K$, the claim \eqref{form50} follows by summing up \eqref{form52} over $1 \leq j \leq K$ and recalling that $KM = N$. 

We next verify that $E_{Q_{0}} \cap Q(\calT_{j}) \subset \cup \mathbf{Leaves}(\calT_{j})$ for all $j \geq 0$, as claimed in property \nref{T2}. Indeed, if $x \in E_{Q_{0}} \cap Q(\calT_{j})$ for some $j \geq 0$, then \eqref{form53} holds, and $Q(\calT_{j})$ is contained in $\leq K$ elements of $\mathbf{Top}$. This means that if $x \in Q(\calT_{j}) \, \setminus \, \cup \mathbf{Leaves}(\calT_{j})$, then $x$ is contained in $\leq K - 1$ distinct leaves, and hence satisfies \eqref{form54}. But this would imply $x \notin E_{Q_{0}}$. Hence $x \in \mathbf{Leaves}(\calT_{j})$, as claimed.

The properties \nref{T2}-\nref{T3} on the list of requirements have now been verified (indeed \nref{T3} holds by the virtue of the stopping condition). For \nref{T1} and \nref{T4}, some further pruning will be needed. First, from \eqref{form50},  \eqref{form51}, and the assumption $\mu(E_{Q_{0}}) \geq \tfrac{1}{2}\mu(Q_{0})$, we infer that
\begin{align*} \frac{N\mu(Q_{0})}{2} & \stackrel{\eqref{form50}}{\leq} \int_{E_{Q_{0}}} \mathop{\sum_{Q \in \calT}}_{\beta(Q) \geq \epsilon} \mathbf{1}_{Q}(x) \, dx\\
& = \sum_{j = 0}^{\infty} \int_{E_{Q_{0}} \cap Q(\calT_{j})} \mathop{\sum_{Q \in \calT_{j}}}_{\beta(Q) \geq \epsilon} \mathbf{1}_{Q}(x) \, dx\\\
& \stackrel{\eqref{form51}}{\leq} M \sum_{j = 0}^{\infty} \mu(E_{Q_{0}} \cap Q(\calT_{j})). \end{align*}
Recalling that $N = KM$, this yields
\begin{displaymath} \sum_{j = 0}^{\infty} \mu(E_{Q_{0}} \cap Q(\calT_{j})) \geq \frac{K\mu(Q_{0})}{2}. \end{displaymath}  
Now, we discard all \emph{light} trees with the property $\mu(E_{Q_{0}} \cap Q(\calT_{j})) < \tfrac{1}{4}\mu(Q(\calT_{j}))$. Then, by the uniform upper bound \eqref{form49}, we have
\begin{displaymath} \sum_{j : \mathcal{T}_{j} \text{ is light}} \mu(E_{Q_{0}} \cap Q(\calT_{j})) \leq \tfrac{1}{4} \sum_{j = 0}^{\infty} \mu(Q(\calT_{j})) \leq \frac{K\mu(Q_{0})}{4}.  \end{displaymath}
Hence, the \emph{heavy trees} with
\begin{displaymath} \mu(E_{Q_{0}} \cap Q(\calT_{j})) \geq \frac{\mu(Q(\calT_{j}))}{4} \end{displaymath}
satisfy
\begin{displaymath} \sum_{j : \mathcal{T}_{j} \text{is heavy}} \mu(Q(\calT_{j})) \geq \frac{K\mu(Q_{0})}{4}. \end{displaymath}
By definition of the heavy trees, the requirements \nref{T1} and \nref{T4} on our list are satisfied (and \nref{T2}-\nref{T3} were not violated by the final pruning, since they are statements about individual trees). After another re-indexing, this completes the construction of the heavy trees $\calT_{0},\calT_{1},\ldots$

We have now proven Proposition \ref{mainProp} modulo Proposition \ref{treeProp}, which concerns an individual heavy tree $\calT_{j}$. Proving Proposition \ref{treeProp} will occupy the rest of the paper.

%%%%%%%%%%%%%%%%%%%%%

\section{A criterion for positive width}\label{s:width}

Let $E \subset \R^{d}$ be a closed $n$-regular set, write $\mu := \calH^{n}|_{E}$, and let $\mathcal{D}$ be a system of dyadic cubes on $E$. I next discuss the notion of width, which appeared in the statement of Proposition \ref{treeProp}. Width was first introduced in \cite{MR4127898} in the context of Heisenberg groups, and \cite[\S 8]{MR4127898} contains the relevant definitions adapted to $\R^{n}$, but only in the case $n = d - 1$. I start here with the higher co-dimensional generalisation.
\begin{definition}[Measure on the affine Grassmannian] Fix $0 < m < d$, and let $\mathcal{A} := \mathcal{A}(d,m)$ be the collection of all affine planes of dimension $m$. Define a measure $\lambda := \lambda_{d,m}$ on $\mathcal{A}$ via the relation
\begin{displaymath} \int_{\mathcal{A}} f(V) \, d\lambda(V) := \int_{G(d,d - m)} \int_{V} f(\pi_{V}^{-1}\{w\}) \, d\calH^{d - m}(w) \, d\gamma_{d,d - m}(V), \qquad f \in C_{c}(\mathcal{A}). \end{displaymath}
\end{definition}

The definition above is standard, see \cite[\S3.16]{zbMATH01249699}. We are interested in the case $m = d - n$, since we plan to slice sets by the fibres of projections to planes in $G(d,n)$.

\begin{definition}[Width]\label{width} For $Q \in \mathcal{D}$ and a plane $W \in \mathcal{A}(d,d - n)$, we define
\begin{displaymath} \w_{Q}(E,W) := \diam(B_{Q} \cap E \cap W), \end{displaymath}
where we recall that $B_{Q} = B(c_{Q},C\ell(Q))$ is a ball centred at some point $c_{Q} \in Q \subset E$ containing $Q$. Then, we also define
\begin{align} \w(Q) & := \frac{1}{\mu(Q)} \int_{\mathcal{A}(d,d - n)} \frac{\w_{Q}(E,W)}{\ell(Q)} \, d\lambda_{d,d - n}(W) \notag\\
&\label{d:width} = \frac{1}{\mu(Q)} \int_{G(d,n)} \int_{V} \frac{\w_{Q}(E,\pi_{V}^{-1}\{w\})}{\ell(Q)} \, d\calH^{n}(w) \, d\gamma_{d,n}(V). \end{align} 
Finally, if $\calF \subset \calD$ is an arbitrary collection of dyadic cubes, we set
\begin{equation}\label{s:width2} \w(\mathcal{F}) := \sum_{Q \in \mathcal{F}} \w(Q)\mu(Q). \end{equation}
\end{definition}
The $\mu(Q)$-normalisation in \eqref{d:width} is the right one, because for $V \in G(d,n)$ fixed, it is only possible that $\w_{Q}(E,\pi_{V}^{-1}\{w\}) \neq 0$ if $w \in \pi_{V}(B_{Q}) \subset V$, and $\calH^{n}(\pi_{V}(B_{Q})) \sim \mu(Q)$. As shown in \cite[Theorem 8.8]{MR4127898}, width satisfies a Carleson packing condition. However, the proof in \cite{MR4127898} was restricted to the case $d = n - 1$, and a little graph-theoretic construction is needed in the higher co-dimensional situation. Details follow.
\begin{proposition}\label{widthCarleson} There exists a constant $C \geq 1$, depending only on the $1$-regularity constant of $E$, such that
\begin{equation}\label{wCarleson} \w(\calD(Q_{0})) \leq C\mu(Q_{0}), \qquad Q_{0} \in \calD, \end{equation}
where $\mathcal{D}(Q_{0}) := \{Q \in \mathcal{D} : Q \subset Q_{0}\}$.
\end{proposition}

\begin{proof} Fix $Q_{0} \in \calD$. By definitions,
\begin{equation}\label{form86} \w(\calD(Q_{0})) = \int_{G(d,n)} \int_{V} \sum_{Q \in \calD(Q_{0})} \frac{\diam(B_{Q} \cap E \cap \pi_{V}^{-1}\{w\})}{\ell(Q)} \, d\calH^{n}(w) \, d\gamma_{d,n}(V). \end{equation} 
The main tool in the proof is \emph{Eilenberg's inequality}
\begin{equation}\label{eilenberg} \int_{V} \card(A \cap \pi_{V}^{-1}\{w\}) \, d\calH^{n}(w) \lesssim_{n} \calH^{n}(A), \qquad V \in G(d,n), \end{equation}
where $A \subset \R^{d}$ is Borel, see \cite[Theorem 7.7]{zbMATH01249699}. In particular, we infer from \eqref{eilenberg} that
\begin{displaymath} q_{V,w} := \card(B_{Q_{0}} \cap E \cap \pi_{V}^{-1}\{w\}) < \infty \end{displaymath}
for all $V \in G(d,n)$ and for $\calH^{n}$ a.e. $w \in V$. We continue our estimate of \eqref{form86} for a fixed plane $V \in G(d,n)$, and for any $w \in V$ such that $q := q_{V,w} < \infty$. If $q \in \{0,1\}$, then
\begin{displaymath} \diam(B_{Q} \cap E \cap \pi_{V}^{-1}\{w\}) \leq \diam(B_{Q_{0}} \cap E \cap \pi_{V}^{-1}\{w\}) = 0, \qquad Q \in \calD(Q_{0}), \end{displaymath}
so these pairs $(V,w)$ contribute nothing to the integral in \eqref{form86}. So, assume that $q \geq 2$, and enumerate the points in $B_{Q_{0}} \cap E \cap \pi_{V}^{-1}\{w\}$ as
\begin{displaymath} B_{Q_{0}} \cap E \cap \pi_{V}^{-1}\{w\} = \{x_{1},\ldots,x_{q}\}. \end{displaymath}
We will next need to construct a "spanning graph" whose vertices are the points $x_{1},\ldots,x_{q}$, and whose edges "$\mathcal{E}$" are a (relatively small) subset of the $\sim q^{2}$ segments connecting the vertices. More precisely, we need the following properties from $\mathcal{E}$:
\begin{enumerate}
\item[(E1) \phantomsection\label{E1}] $\card \mathcal{E} \lesssim_{d} q$.
\item[(E2) \phantomsection\label{E2}] For every $1 \leq i < j \leq q$, there is a connected union of edges in $\mathcal{E}$ which connects $x_{i}$ to $x_{j}$ inside $\bar{B}(x_{i},2|x_{i} - x_{j}|)$. 
\end{enumerate} 
Property \nref{E2} sounds like quasiconvexity, but is weaker: there are no restrictions on the length of the connecting $\mathcal{E}$-path, as long as it is contained in $B(x_{i},2|x_{i} - x_{j}|)$. Let us then find the edges with the properties \nref{E1}-\nref{E2}. Let $\xi_{1},\ldots,\xi_{p} \subset S^{d - 1}$ be a maximal $\tfrac{1}{4}$-separated set on $S^{d - 1}$, with $p \sim_{d} 1$, and let 
\begin{displaymath} C_{j} := \{re : e \in B(\xi_{j},\tfrac{1}{2}) \cap S^{d - 1} \text{ and } r > 0\}, \qquad 1 \leq j \leq p, \end{displaymath}
be a directed open cone around the half-line $\{r\xi_{j} : r > 0\}$. By the net property of $\xi_{1},\ldots,\xi_{p}$,
\begin{equation}\label{form87} \R^{d} \, \setminus \, \{0\} \subset \bigcup_{j = 1}^{p} C_{j}. \end{equation}
We claim that the following holds: if $y \in x + C_{j}$, then 
\begin{equation}\label{form88} \bar{B}(x,|x - y|) \cap (x + C_{j}) \subset B(y,|x - y|). \end{equation} 
First, use translations and dilations to reduce to the case $x = 0$ and $|x - y| = 1$:
\begin{displaymath} y \in C_{j} \cap S^{d - 1} \quad \Longrightarrow \quad \bar{B}(1) \cap C_{j} \subset B(y,1). \end{displaymath} 
To check this case, one first verifies by explicit computation that if $y \in S^{d - 1}$, then the set $C_{y} := \{re : e \in B(y,1) \cap S^{d - 1} \text{ and } 0 < r \leq 1\}$ is contained in $B(y,1)$. Consequently, 
\begin{displaymath} y \in C_{j} \cap S^{d - 1} \subset B(\xi_{j},\tfrac{1}{2}) \quad \Longrightarrow \quad B(\xi_{j},\tfrac{1}{2}) \subset B(y,1) \quad \Longrightarrow \quad \bar{B}(1) \cap C_{j} \subset C_{y} \subset B(y,1). \end{displaymath}

We are then prepared to define the edge set $\mathcal{E}$. Fix one of the points $x_{i}$, $1 \leq i \leq q$. For every of $1 \leq j \leq p$, draw an edge (that is, a segment) between $x_{i}$ and one of the points closest to $x_{i}$ in the finite set
\begin{displaymath} \{x_{1},\ldots,x_{q}\} \cap (x_{i} + C_{j}) \subset \{x_{1},\ldots,x_{q}\} \, \setminus \, \{x_{i}\}, \end{displaymath}
if the intersection on the left hand side is non-empty; this is the case for at least one $j \in \{1,\ldots,p\}$ by \eqref{form87}. Thus, for every $x_{i}$, one draws $\sim_{d} 1$ edges. Let $\mathcal{E}$ be the collection of all edges so obtained. Then $\card \mathcal{E} \sim_{d} q$, so requirement \nref{E1} is met. 

To prove \nref{E2}, fix $s_{0} := x_{i}$ and $t := x_{j}$ with $1 \leq i < j \leq q$. The plan is to find, recursively, a collection of segments $I_{j} := [s_{j - 1},s_{j}] \in \mathcal{E}$, $1 \leq j \leq k$, whose union is connected, contains $\{s_{0},t\}$ (indeed $s_{k} = t$) and is contained in 
\begin{displaymath} \bar{B}(t,|s_{0} - t|) \subset \bar{B}(s,2|s_{0} - t|). \end{displaymath}
By \eqref{form87}, there is a half-cone $C_{j_{1}}$ with $t \in s_{0} + C_{j_{1}}$. Let $I_{1} = [s_{0},s_{1}] \in \mathcal{E}$ be the edge connecting $s_{0}$ to one of the nearest points $s_{1} \in \{x_{1},\ldots,x_{q}\} \cap (s_{0} + C_{j_{1}})$. Evidently $|s_{0} - s_{1}| \leq |s_{0} - t|$, since $t \in \{x_{1},\ldots,x_{q}\} \cap (s_{0} + C_{j_{1}})$ itself is one of the candidates among which $s_{1}$ is chosen. Hence, applying \eqref{form88} with $x = s_{0}$ and $y = t$, we find that
\begin{equation}\label{form89} s_{1} \in \bar{B}(s_{0},|s_{0} - t|) \cap (s_{0} + C_{j_{1}}) \subset B(t,|s_{0} - t|). \end{equation}
In particular,
\begin{equation}\label{form90} |s_{1} - t| < |s_{0} - t|. \end{equation}
Also, we see from \eqref{form89} that $\partial I_{1} = \{s_{0},s_{1}\} \subset \bar{B}(t,|s_{0} - t|)$, and hence $I_{1} \subset \bar{B}(t,|s_{0} - t|)$ by convexity. We then replace "$s_{0}$" by "$s_{1}$" and repeat the procedure above: by \eqref{form87}, there is a half-cone $C_{j_{2}}$ with the property $t \in s_{1} + C_{j_{2}}$ (unless $s_{1} = t$ and we are done already), and we let $I_{2} = [s_{1},s_{2}] \in \mathcal{E}$ be the edge connecting $s_{1}$ to the nearest point $s_{2} \in \{x_{1},\ldots,x_{q}\} \cap (s_{1} + C_{j_{2}})$. Then $|s_{1} - s_{2}| \leq |s_{1} - t|$ (otherwise we chose $t$ over $s_{2}$), so
\begin{displaymath} s_{2} \in \bar{B}(s_{1},|s_{1} - t|) \cap (s_{1} + C_{j_{2}}) \stackrel{\eqref{form88}}{\subset} B(t,|s_{1} - t|) \stackrel{\eqref{form90}}{\subset} \bar{B}(t,|s_{0} - t|).  \end{displaymath} 
From the inclusions above, we infer that $I_{2} \subset \bar{B}(t,|s_{0} - t|)$, and also 
\begin{displaymath} |s_{2} - t| < |s_{1} - t| \stackrel{\eqref{form90}}{<} |s - t|. \end{displaymath} 
We proceed inductively, finding further segments $[s_{i},s_{i + 1}] \in \mathcal{E}$, which are contained in $\bar{B}(t,|s_{0} - t|)$, and with the property that $|s_{j + 1} - t| < |s_{j} - t| < \ldots < |s_{0} - t|$. Since the points $s_{j}$ are drawn from the finite set $\{x_{1},\ldots,x_{q}\}$, these strict inequalities eventually force $s_{k} = t$ for some $k \geq 1$, and at that point the proof of property \nref{E2} is complete. 

Let us then use the edges $\mathcal{E}$ constructed above to estimate the integrand in \eqref{form86}. I claim that
\begin{equation}\label{form91} \sum_{Q \in \calD(Q_{0})} \frac{\diam(B_{Q} \cap E \cap \pi_{V}^{-1}\{w\})}{\ell(Q)} \lesssim \sum_{I \in \mathcal{E}} \mathop{\sum_{Q \in \calD(Q_{0})}}_{I \subset 4B_{Q}} \frac{|I|}{\ell(Q)}. \end{equation} 
To see this, fix $Q \in \mathcal{D}(Q_{0})$, and let $x_{i},x_{j} \in B_{Q} \cap E \cap \pi_{V}^{-1}\{w\} \subset \{x_{1},\ldots,x_{q}\}$ be points such that 
\begin{displaymath} |x_{i} - x_{j}| = \diam(B_{Q} \cap E \cap \pi_{V}^{-1}\{w\}). \end{displaymath} 
According to property \nref{E2} of the edge family $\mathcal{E}$, there exists a connected union of segments in $\mathcal{E}$ which is contained in
\begin{displaymath} B(x_{i},2|x_{i} - x_{j}|) \subset 4B_{Q} \end{displaymath} 
and which contains $\{x_{i},x_{j}\}$. Since the union is connected, the total length of the segments involved exceeds $|x_{i} - x_{j}|$:
\begin{displaymath} \diam(B_{Q} \cap E \cap \pi_{V}^{-1}\{w\}) = |x_{i} - x_{j}| \leq \mathop{\sum_{I \in \mathcal{E}}}_{I \subset 4B_{Q}} |I|. \end{displaymath}
Swapping the order of summation proves \eqref{form91}. To complete the proof of the proposition, fix $I \in \mathcal{E}$, and consider the inner sum in \eqref{form91}. Note that the inclusion $I \subset 4B_{Q}$ is only possible if $\ell(Q) \gtrsim |I|$. On the other hand, for a fixed side-length $2^{-j} \gtrsim |I|$, there are $\lesssim 1$ cubes $Q \in \mathcal{D}(Q_{0})$ with $\ell(Q) = 2^{-j}$ and $I \subset 4B_{Q}$. Putting these observations together,
\begin{displaymath} \mathop{\sum_{Q \in \calD(Q_{0})}}_{I \subset 4B_{Q}} \frac{|I|}{\ell(Q)} \lesssim 1. \end{displaymath} 
From this, \eqref{form91}, and the cardinality estimate $\card \mathcal{E} \lesssim_{d} q$ from \nref{E1} it follows that
\begin{displaymath} \sum_{Q \in \mathcal{D}(Q_{0})} \frac{\diam(B_{Q} \cap E \cap \pi_{V}^{-1}\{w\})}{\ell(Q)} \lesssim \card \mathcal{E} \lesssim_{d} q = \card (B_{Q_{0}} \cap E \cap \pi_{V}^{-1}\{w\}). \end{displaymath} 
Plugging this estimate into \eqref{form86} and using Eilenberg's inequality \eqref{eilenberg}, one finds that
\begin{displaymath} \w(\calD(Q_{0})) \lesssim_{d} \int_{G(d,n)} \int_{V} \card (B_{Q_{0}} \cap E \cap \pi_{V}^{-1}\{w\}) \, d\calH^{n}(w) \, d\gamma_{d,n}(V) \lesssim \mu(Q_{0}). \end{displaymath}
This completes the proof of the proposition. \end{proof}

Recall that our objective, in Proposition \ref{treeProp}, is to prove that each heavy tree $\calT_{j}$ satisfies $\w(\calT_{j}) \gtrsim \mu(Q(\calT_{j}))$ if the parameter $M \geq 1$ was chosen large enough. To accomplish this, we start by recording a technical criterion which guarantees that a general tree $\calT \subset \calD$ satisfies $\w(\calT) \gtrsim \mu(Q(\calT))$. Afterwards, the criterion will need to be verified for heavy trees.  

\begin{proposition}\label{prop2} For every $c,\delta > 0$ and $C_{0} \geq 1$ there exists $N \geq 1$ such that the following holds. Assume that the $n$-regularity constant of $E$ is at most $C_{0}$. Let $\mathcal{T} \subset \mathcal{D}$ be a tree with top cube $Q_{0} := Q(\mathcal{T})$. Assume that there is a subset $\mathcal{G} \subset \mathbf{Leaves}(\mathcal{T})$ with the following properties.
\begin{itemize}
\item All the cubes in $\mathcal{G}$ have \textup{PBP} with common plane $V_{0} \in G(d,n)$ and constant $\delta$:
\begin{equation}\label{form75} \calH^{n}(\pi_{V}(E \cap B_{Q})) \geq \delta \mu(Q), \qquad Q \in \mathcal{G}, \, V \in B(V_{0},\delta). \end{equation}
\item Write $f_{V} := \sum_{Q \in \mathcal{G}} \mathbf{1}_{\pi_{V}(B_{Q})}$ for $V \in B(V_{0},\delta)$. Assume that there is a subset $S_{G} \subset B(V_{0},\delta)$ such that the "high multiplicity" sets $H_{V} := \left\{x \in V : \mathcal{M}f_{V}(x) \geq N\right\}$ satisfy
\begin{equation}\label{form19} \int_{H_{V}} f_{V}(x) \, dx \geq cN^{-1}\mu(Q_{0}), \qquad V \in S_{G}. \end{equation} 
\end{itemize}
Here $\mathcal{M}f_{V}$ is the (centred) Hardy-Littlewood maximal function of $f_{V}$. Then 
\begin{displaymath} \w(\calT) \gtrsim c\delta N^{-1} \mu(Q_{0}) \cdot \gamma_{d,n}(S_{G}), \end{displaymath}
where the implicit constant only depends on "$d$" and the $n$-regularity constant of $E$.
\end{proposition}

The proof of Proposition \ref{prop2} would be fairly simple if all the leaves in $\mathcal{G}$ had approximately the same generation in $\calD$. In our application, this cannot be assumed, unfortunately, and we will need another auxiliary result to deal with the issue: 

\begin{lemma}\label{lemma1} Fix $M,d,\gamma \geq 1$ and $c > 0$. Then, the following holds if $A = A_{d} \geq 1$ is large enough, depending only on $d$ (as in "$\R^{d}$"), and 
\begin{equation}\label{form41} N > A^{(\gamma + 1)^{2}}M^{\gamma + 2}/c \end{equation} 
Let $\mathcal{B}$ be a collection of balls contained in $B(0,1) \subset \R^{d}$, and associate to every $B \in \mathcal{B}$ a weight $w_{B} \geq 0$. Set
\begin{displaymath} f = \sum_{B \in \mathcal{B}} w_{B}\mathbf{1}_{B}, \end{displaymath}
and write $H_{N} := \{\mathcal{M}f \geq N\}$, where $\mathcal{M}f$ is the Hardy-Littlewood maximal function of $f$. Assume that
\begin{displaymath} \int_{H_{N}} f(x) \, dx \geq cN^{-\gamma}, \end{displaymath}
Then, there exists a collection $\mathcal{R}_{\mathrm{heavy}}$ of disjoint cubes such that the "sub-functions"
\begin{displaymath} f_{R} := \mathop{\sum_{B \in \mathcal{B}}}_{B \subset R} w_{B}\mathbf{1}_{B}, \qquad R \in \mathcal{R}_{\mathrm{heavy}}, \end{displaymath}
satisfy the following properties:
\begin{displaymath} \sum_{R \in \mathcal{R}_{\mathrm{heavy}}} \|f_{R}\|_{1} \geq c2^{-2(\gamma + 1)}N^{-\gamma} \quad \text{and} \quad \|f_{R}\|_{1} > M|R|, \qquad R \in \mathcal{R}_{\mathrm{heavy}}. \end{displaymath} 
\end{lemma}

The lemma is easy in the case where the balls in $\mathcal{B}$ have common radius, say $r$. Then one can take $\mathcal{R}_{\mathrm{heavy}}$ to be a suitable collection of disjoint cubes of side-length $\sim r$. In the application to Proposition \ref{prop2}, this case corresponds to the situation where $\ell(Q) \sim \ell(Q')$ for all $Q,Q' \in \mathcal{G}$. In the general case, the elementary but lengthy proof of Lemma \ref{lemma1} is contained in Appendix \ref{appA}.

We then prove Proposition \ref{prop2}, taking Lemma \ref{lemma1} for granted:
 
\begin{proof}[Proof of Proposition \ref{prop2}] The plan is to show that
\begin{equation}\label{form21} \sum_{Q \in \mathcal{T}} \int_{V} \frac{\w_{Q}(E,\pi_{V}^{-1}\{w\})}{\ell(Q)} \, d\calH^{n}(w) \gtrsim c\delta N^{-1}\mu(Q_{0}), \qquad V \in S_{G}. \end{equation}
The proposition then follows by recalling the definitions of $\w(Q)$ and $\w(\calT)$ from \eqref{d:width}-\eqref{s:width2} and integrating \eqref{form21} over $V \in S_{G}$.

To prove \eqref{form21}, we assume, to avoid a rescaling argument, that $\ell(Q_{0}) = 1$. Then, we begin by re-interpreting \eqref{form19} in such a way that we may apply Lemma \ref{lemma1}. Namely, we identify $V \in S_{G}$ with $\R^{n}$, and consider the collection of balls 
\begin{displaymath} \mathcal{B} := \{\pi_{V}(B_{Q}) : Q \in \mathcal{G}\}. \end{displaymath}
More precisely, let $\calB$ be an index set for the balls $\pi_{V}(B_{Q})$ such that if some ball $B = \pi_{V}(B_{Q})$ arises from multiple distinct cubes $Q \in \mathcal{G}$, then $B$ has equally many indices in $\mathcal{B}$. 

Note that the balls in $\mathcal{B}$ are all contained in 
\begin{displaymath} B_{0} := \pi_{V}(B_{Q_{0}}), \end{displaymath}
since $B_{Q} \subset B_{Q'}$ whenever $Q,Q' \in \mathcal{D}$ and $Q \subset Q'$. We then define $f := \sum_{B \in \mathcal{B}} \mathbf{1}_{B}$ and $H_{N} := \{x \in V : \mathcal{M}f(x) \geq N\}$. It follows from \eqref{form19}, and the assumption $\ell(Q_{0}) = 1$, that
\begin{displaymath} \int_{H_{N}} f(w) \, dw \gtrsim cN^{-1}. \end{displaymath}
In other words, the hypotheses of Lemma \ref{lemma1} are met with $\gamma = 1$. We fix $M := C\delta^{-1}$, where $C \geq 1$ is a large constant to be specified soon, depending only on the $n$-regularity constant of $E$. We then assume that $N > AM^{3}/c$, in accordance with \eqref{form41}. Lemma \ref{lemma1} now provides us with a collection $\mathcal{R} = \mathcal{R}_{\mathrm{heavy}}$ of disjoint cubes in $\R^{n} \cong V$ such that
\begin{equation}\label{form22} \sum_{R \in \mathcal{R}} \|f_{R}\|_{1} \gtrsim cN^{-1} \quad \text{and} \quad \|f_{R}\|_{1} \geq M|R| \text{ for } R \in \mathcal{R}. \end{equation} 
In this proof we abbreviate $|\cdot| := \calH^{n}|_{V}$. We recall that
\begin{displaymath} f_{R} = \mathop{\sum_{B \in \mathcal{B}}}_{B \subset R} w_{B}\mathbf{1}_{B} = \mathop{\sum_{Q \in \mathcal{G}}}_{B_{Q} \subset T(R)} \mathbf{1}_{\pi_{V}(B_{Q})}, \end{displaymath}
where $T(R) := \pi_{V}^{-1}(R)$. Therefore, the conditions in \eqref{form22} are equivalent to
\begin{equation}\label{form24} \sum_{R \in \mathcal{R}} \mathop{\sum_{Q \in \mathcal{G}}}_{B_{Q} \subset T(R)} \mu(Q) \gtrsim cN^{-1} \quad \text{and} \quad \mathop{\sum_{Q \in \mathcal{G}}}_{B_{Q} \subset T(R)} \mu(Q) \gtrsim M|R|, \quad R \in \mathcal{R}, \end{equation} 
where the implicit constants depend on the $n$-regularity constant of $\mu$. We now make a slight refinement to the set $\mathcal{G}$: for $R \in \mathcal{R}$ fixed, we apply the $5r$-covering theorem to the balls $\{2B_{Q} : Q \in \mathcal{G} \text{ and } B_{Q} \subset T(R)\}$. As a result, we obtain a sub-collection $\mathcal{G}_{R} \subset \mathcal{G}$ with the properties
\begin{equation}\label{form25} 2B_{Q} \cap 2B_{Q'} = \emptyset, \qquad Q,Q' \in \mathcal{G}_{R}, \, Q \neq Q', \end{equation}
and
\begin{displaymath} \mathop{\bigcup_{Q \in \mathcal{G}}}_{B_{Q} \subset T(R)} Q \subset \bigcup_{Q \in \mathcal{G}_{R}} 10B_{Q}. \end{displaymath}
In particular, by \eqref{form24},
\begin{equation}\label{form23a} \sum_{R \in \mathcal{R}} \sum_{Q \in \mathcal{G}_{R}} \mu(Q) \gtrsim \sum_{R \in \mathcal{R}} \mathop{\sum_{Q \in \mathcal{G}}}_{B_{Q} \subset T(R)} \mu(Q) \gtrsim cN^{-1} \end{equation}
and
\begin{equation}\label{form23} \sum_{Q \in \mathcal{G}_{R}} \mu(Q) \sim \sum_{Q \in \mathcal{G}_{R}} \mu(10B_{Q}) \gtrsim M|R| \end{equation}
by \eqref{form24}. We also write $\mathcal{B}_{R} := \{\pi_{V}(B_{Q}) : Q \in \mathcal{G}_{R}\}$, $R \in \mathcal{R}$, so $\mathcal{B}_{R} \subset \mathcal{B}$ is a collection of balls contained in $R$ satisfying
\begin{equation}\label{form43} \sum_{B \in \mathcal{B}_{R}} |B| \gtrsim M|R|, \qquad R \in \mathcal{R}. \end{equation}
Just like $\mathcal{B}$, the set $\calB_{R}$ should also, to be precise, be defined as a set of indices, accounting for the possibility that $B = \pi_{V}(B_{Q})$ arises from multiple cubes $Q \in \mathcal{G}_{R}$. Next, recall a key assumption of the proposition, namely that all the cubes in $\mathcal{G}$ have PBP with common ball $B(V_{0},\delta) \subset G(d,n)$. In particular, for our fixed plane $V \in S_{G} \subset B(V_{0},\delta)$, we have
\begin{equation}\label{form42} \calH^{n}(\pi_{V}(B_{Q} \cap E)) \geq \delta \mu(Q), \qquad Q \in \mathcal{G}. \end{equation}
Since the balls $B_{Q}$, $Q \in \mathcal{G}$, are all contained in $B_{0} := B_{Q_{0}}$, the ball associated with the top cube of the tree, the conclusion of \eqref{form42} persists if we replace $B_{Q} \cap E$ by $B_{Q} \cap E \cap B_{0}$. For $B = \pi_{V}(B_{Q})$ with $Q \in \mathcal{G}$, write $E_{B} := \pi_{V}(B_{Q} \cap E \cap B_{0})$, so \eqref{form42} implies that $|E_{B}| \gtrsim \delta |B|$. Then, for $R \in \mathcal{R}$ fixed, we infer from \eqref{form43} that
\begin{displaymath} \int_{R} \sum_{B \in \mathcal{B}_{R}} \mathbf{1}_{E_{B}}(w) \, dw = \sum_{B \in \mathcal{B}_{R}} |E_{B}| \gtrsim \delta \sum_{B \in \mathcal{B}_{R}} |B| \gtrsim \delta M|R| = C|R|. \end{displaymath} 
We now choose the constant $C \geq 1$ so large that
\begin{equation}\label{form44a} \int_{R} \sum_{B \in \mathcal{B}_{R}} \mathbf{1}_{E_{B}}(w) \, dw \geq 2|R|. \end{equation} 
Then, if we consider the "set of multiplicity $\leq 1$",
\begin{displaymath} L_{R} := \left\{w \in R : \sum_{B \in \mathcal{B}_{R}} \mathbf{1}_{E_{B}}(w) \leq 1 \right\} \subset R, \end{displaymath} 
we may infer from \eqref{form44a} that
\begin{displaymath} \int_{L_{R}} \sum_{B \in \mathcal{B}_{R}} \mathbf{1}_{E_{B}}(w) \, dw \leq |R| \leq \frac{1}{2} \int_{R} \sum_{B \in \mathcal{B}_{R}} \mathbf{1}_{E_{B}}(w) \, dw. \end{displaymath}
Consequently, if $P_{R} := R \, \setminus \, L_{R}$ is the "positive multiplicity set", we have
\begin{equation}\label{form44} \int_{P_{R}} \sum_{B \in \mathcal{B}_{R}} \mathbf{1}_{E_{B}}(w) \, dw \geq \frac{1}{2} \int_{R} \sum_{B \in \mathcal{B}_{R}} \mathbf{1}_{E_{B}}(w) \, dw \gtrsim \delta \sum_{Q \in \mathcal{G}_{R}} \mu(Q). \end{equation} 
Fix $w \in P_{R} \subset R$, and write
\begin{displaymath} m := m_{w} := \sum_{B \in \mathcal{B}_{R}} \mathbf{1}_{E_{B}}(w) \geq 2. \end{displaymath}
(If the sum happens to equal $\infty$, pick $m \geq 2$ arbitrary; eventually one will have to let $m \to \infty$ in this case). Unraveling the definitions, the $(d - n)$-plane $W := W_{w} := \pi_{V}^{-1}\{w\}$ contains $m$ points of $E \cap B_{0}$ inside $m$ distinct balls $B_{Q}$, with $Q \in \mathcal{G}_{R}$. Let $P \subset E \cap W$ be the set of these $m$ points, and define the following set $\mathcal{E}$ of edges connecting (some) pairs of points in $P$: for every point $p \in P$, pick exactly one of the points $q \in P \, \setminus \, \{p\}$ at minimal distance from $p$, and add the edge $(p,q)$ to $\mathcal{E}$. Note that $\card \mathcal{E} = m$, since $\mathcal{E}$ contains precisely one edge of the form $(p,q)$ for every $p \in P$. We have now used the assumption $m \geq 2$: otherwise we could not have drawn any edges in the preceding manner! Note that the edges in the graph $(P,\mathcal{E})$ are directed: $(p,q) \in \mathcal{E}$ does not imply $(q,p) \in \mathcal{E}$.

Now that the edge set $\mathcal{E}$ has been constructed, define the following relation between edges $I \in \mathcal{E}$ and the cubes $Q \in \mathcal{T}$: write $I \prec Q$ if $I \subset B_{Q}$, and $|I| \geq \rho \ell(Q)$. Slightly abusing notation, here $I$ also refers to the segment $[p,q]$, for an edge $(p,q) \in \mathcal{E}$. The choice of the constant $\rho > 0$ will become apparent soon, and it will only depend on the $n$-regularity constant of $E$. We now claim that
\begin{equation}\label{form45} \sum_{I \in \mathcal{E}} \mathop{\sum_{Q \in \mathcal{T}}}_{I \prec Q} \frac{|I|}{\ell(Q)} \gtrsim \card \mathcal{E} = m. \end{equation}
We already know that $\card \mathcal{E} = m$, so it remains to prove the first inequality. Fix $I = (p,q) \in \mathcal{E}$, with $p,q \in P$. Then, by the definition of $P$, the points $p$ and $q$ are contained in two balls $B_{p} := B_{Q_{p}}$ and $B_{q} := B_{Q_{q}}$, respectively, with $Q_{p},Q_{q} \in \mathcal{G}_{R}$ and $Q_{p} \neq Q_{q}$. In particular, we recall from \eqref{form25} that $2B_{p} \cap 2B_{q} = \emptyset$. Hence $p \notin 2B_{q}$, and $|I| \gtrsim \ell(Q_{q})$. On the other hand, $p,q \in B_{0}$, so $|I| \lesssim \ell(Q_{0})$. Let $Q' \supset Q_{q}$ be the smallest cube in the ancestry of $Q_{q}$ such that $p,q \in B_{Q'}$. Then $Q_{q} \subsetneq Q' \subset Q_{0}$, hence $Q' \in \mathcal{T}$, and
\begin{equation}\label{form92} \ell(Q') \lesssim |I|. \end{equation} 
Since $p,q \in B_{Q'}$, by convexity also $I \subset B_{Q'}$. If the constant "$\rho$" in the definition of "$\prec$" was chosen appropriately, we infer from $I \subset B_{Q'}$ and \eqref{form92} that $I \prec Q'$. This proves the lower bound in \eqref{form45}. 

Next, we claim that
\begin{equation}\label{form46} \w_{Q}(E,\pi_{V}^{-1}\{w\}) = \diam(E \cap B_{Q} \cap W) \gtrsim_{d} \mathop{\sum_{I \in \mathcal{E}}}_{I \prec Q} |I|, \qquad Q \in \mathcal{T}. \end{equation}
Indeed, fix $Q \in \calT$ and assume that there is at least one edge $I \in \mathcal{E}$ such that $I \prec Q$. Then $I \subset B_{Q} \cap W$, and both endpoints of $I$ lie in $E$, so $\diam(E \cap B_{Q} \cap W) \geq |I|$. Thus, \eqref{form46} boils down to showing that $\card \{I \in \mathcal{E} : I \prec Q\} \lesssim_{d} 1$. Let $P_{Q} := \{p \in P : (p,q) \in \mathcal{E} \text{ and } (p,q) \prec Q \text{ for some } q \in P \, \setminus \, \{p\}\}$. Then
\begin{displaymath} \card \{I \in \mathcal{E} : I \prec Q\} \leq \card P_{Q}, \end{displaymath}
since $\mathcal{E}$ contains precisely one edge of the form $(p,q)$ for all $p \in P$, i.e. the map $I = (p,q) \mapsto p$ is injective $\{I \in \mathcal{E} : I \prec Q\} \to P_{Q}$. So, it remains to argue that $\card P_{Q} \lesssim_{d} 1$. Otherwise, if $\card P_{Q} \gg_{d} 1$, there exist two distinct points $p_{1},p_{2} \in P_{Q}$ with $|p_{1} - p_{2}| < \rho \ell(Q)$. However, if $q \in P$ is such that $I := (p_{1},q) \prec Q$, then $|I| \geq \rho \ell(Q)$, and since $(p_{1},q) \in \mathcal{E}$, the point $q$ must be one of the nearest neighbours of $p$ in $P \, \setminus \, \{p\}$. This is not true, however, since $|p_{1} - p_{2}| < |p_{1} - q|$. We have proven \eqref{form46}.

A combination of \eqref{form45} and \eqref{form46} leads to
\begin{equation}\label{form84} \sum_{Q \in \mathcal{T}} \frac{\w_{Q}(E,\pi_{V}^{-1}\{w\})}{\ell(Q)} \geq \sum_{I \in \mathcal{E}} \mathop{\sum_{Q \in \mathcal{T}}}_{I \prec Q} \frac{|I|}{\ell(Q)} \gtrsim m = m_{w}, \quad w \in P_{R}. \end{equation} 
Here $P_{R}$ is the subset of $R$ introduced above \eqref{form44}.  Integrating over $w \in R$ next gives
\begin{displaymath} \int_{R} \sum_{Q \in \calT} \frac{\w_{Q}(E,\pi_{V}^{-1}\{w\})}{\ell(Q)} \, dw \stackrel{\eqref{form84}}{\gtrsim} \int_{P_{R}} m_{w} \, dw := \int_{P_{R}} \sum_{B \in \mathcal{B}_{R}} \mathbf{1}_{E_{B}}(w) \, dw \stackrel{\eqref{form44}}{\gtrsim} \delta \sum_{Q \in \mathcal{G}_{R}} \mu(Q). \end{displaymath} 
Finally, summing the result over the (disjoint) cubes $R \in \mathcal{R}$, and using \eqref{form23a}, we find that
\begin{displaymath} \sum_{Q \in \mathcal{T}} \int_{V} \frac{\w_{Q}(E,\pi_{V}^{-1}\{w\})}{\ell(Q)} \, d\calH^{n}(w) \gtrsim c \delta N^{-1}. \end{displaymath}
This completes the proof of \eqref{form21}, and the proof of the proposition. \end{proof}

%%%%%%%%%%%%%%%%%%%%

\section{From big $\beta$ numbers to heavy cones}\label{s1}

Proposition \ref{prop2} contains criteria for showing that $\w(\calT) \gtrsim \mu(Q(\calT_{j}))$. To prove Proposition \ref{treeProp}, these criteria need to be verified for the heavy trees $\calT_{j}$. The selling points \nref{T1}-\nref{T4} of a heavy tree $\calT_{j}$ were that all of its leaves are contained in $M$ cubes in $\calT_{j}$ with non-negligible $\beta$-number (see \nref{T3}), and the total $\mu$ measure of the leaves is at least $\tfrac{1}{4}\mu(Q(\calT_{j}))$ (see \nref{T1}-\nref{T2}). We will use this information to show that if a reasonably wide cone is centred at a typical point $x$ contained in one of the leaves of $\calT_{j}$, then the cone intersects many other leaves at many different (dyadic) distances from $x$.

We first need to set up our notation for cones:

\begin{definition}[Cones] Let $V_{0} \in G(d,n)$, $\alpha > 0$, and $x \in \R^{d}$. We write
\begin{displaymath} X(x,V,\alpha) = \{y \in \R^{d} : |\pi_{V}(x - y)| \leq \alpha |x - y|\}. \end{displaymath}
For $0 < r < R < \infty$, we also define the truncated cones
\begin{displaymath} X(x,V,\alpha,r,R) := X(x,V,\alpha) \cap \bar{B}(x,R) \, \setminus \, B(x,r). \end{displaymath}
\end{definition}
Note the non-standard notation: $X(x,V,\alpha)$ is a cone with axis $V^{\perp} \in G(d,d - n)$! The next proposition extracts "conical" information from many big $\beta$-numbers:

\begin{proposition}\label{prop1} Let $\alpha,d,\epsilon,\theta > 0$ and $C_{0},H \geq 1$. Then, there exists $M \geq 1$, depending only on the previous parameters, such that the following holds. Let $E_{0} \subset \R^{d}$ be a $n$-regular set with regularity constant at most $C_{0}$, and let $E \subset E_{0} \cap B(0,1)$ be a subset of measure $\calH^{n}(E) \geq \theta > 0$ with the following property: for every $x \in E$, there exist $M$ distinct dyadic scales $0 < r < 1$ such that 
\begin{displaymath} \beta(B(x,r)) := \beta_{E_{0}}(B(x,r)) := \inf_{V \in \mathcal{A}(d,n)} \frac{1}{r^{n}} \int_{B(x,r) \cap E_{0}} \frac{\dist(x,V)}{r} \, d\calH^{n}(x) \geq \epsilon. \end{displaymath}
Then, there exists a subset $G \subset E$ of measure $\calH^{1}(G) \geq \theta/2$ such that for all $x \in G$,
\begin{equation}\label{form1} \card \{j \geq 0 : X(x,V,\alpha,2^{-j - 1},2^{-j}) \cap E \neq \emptyset\} \geq H \text{ for all } V \in G(d,n). \end{equation} 
\end{proposition}

The key point of Proposition \ref{prop1} is that information about the $\beta$-numbers relative to the "ambient" set $E_{0}$ is sufficient to imply something useful about cones intersecting the subset $E$. The proof is heavily based on \cite[Proposition 1.12]{MR3790063}, which we quote here:
\begin{proposition}\label{MOProp} Let $\alpha,d,\theta > 0$ and $C_{0},H \geq 1$. Then, there exist constants $\tau > 0$ and $L \geq 1$, depending only on the previous parameters, such that the following holds. Let $E_{0} \subset \R^{d}$ be an $n$-regular set with regularity constant at most $C_{0}$, and let $B \subset E_{0} \cap B(0,1)$ be a subset with $\calH^{n}(B) \geq \theta$ satisfying the following: there exists $V \in G(d,n)$ such that for every $x \in B$, 
\begin{displaymath} \card \{j \geq 0 : X(x,V,\alpha,2^{-j - 1},2^{-j}) \cap B \neq \emptyset\}  \leq H. \end{displaymath} 
Then, there exists a subset $B' \subset B$ with $\calH^{n}(B') \geq \tau$ which is contained on an $L$-Lipschitz graph over $V$. In fact, one can take $L \sim 2^{H}/\alpha$.
\end{proposition}

We may then prove Proposition \ref{prop1}.
\begin{proof}[Proof of Proposition \ref{prop1}] It suffices to show that the subset $B \subset E$ such that \eqref{form1} fails has measure $\calH^{n}(B) < \theta/2$ if $M \geq 1$ was chosen large enough. Assume to the contrary that $\calH^{n}(B) \geq \theta/2$. By definition, for every $x \in B$, there exists an plane $V_{x} \in G(d,n)$ such that
\begin{equation}\label{form2} \card \{j \geq 0 : X(x,V_{x},\alpha,2^{-j - 1},2^{-j}) \cap E \neq \emptyset\}  < H. \end{equation}
We observe that the dependence of $V_{x}$ on $x \in B$ can be removed, at the cost of making $B$ and $\alpha$ slightly smaller. Indeed, choose an $\tfrac{\alpha}{2}$-net $V_{1},\ldots,V_{k} \subset G(d,n)$ with $k \sim_{\alpha,d,n} 1$, and note that for every $x \in B$, there exists $1 \leq j \leq k$ such that 
\begin{displaymath} \card \{i \geq 0 : X(x,V_{j},\tfrac{\alpha}{2},2^{-i - 1},2^{-i}) \cap E \neq \emptyset\}  < H. \end{displaymath}
By the pigeonhole principle, there is a subset $B' \subset B$ of measure $\calH^{n}(B') \gtrsim_{\alpha,d,n} \calH^{n}(B) \geq \theta/2$ such that the choice of $V := V_{j}$ is common for $x \in B'$. It follows that \eqref{form2} holds for this $V$, for all $x \in B'$, with $\tfrac{\alpha}{2}$ in place of $\alpha$. We replace $B$ by $B'$ without altering notation, that is, we assume that \eqref{form2} holds for all $x \in B$, and for some fixed $V \in G(d,n)$.

Now Proposition \ref{MOProp} can be applied to the set $B$, and the plane $V$. The conclusion is that there is a further subset $B' \subset B$ of measure 
\begin{equation}\label{form8} \calH^{n}(B') \sim_{\alpha,d,C_{0},\theta,H} 1 \end{equation}
which is contained in $\Gamma \cap B(0,1)$, where $\Gamma \subset \R^{d}$ is an $L$-Lipschitz graph over $V$ for some $L \sim 2^{H}/\alpha \sim_{\alpha,H} 1$. We will derive a contradiction, using that $B' \subset E$ and, consequently, 
\begin{equation}\label{form3} \beta_{E_{0}}(B(x,r)) \geq \epsilon \end{equation}
for all $x \in B'$, and for $M$ distinct dyadic scales $0 < r < 1$ (which may depend on $x \in B'$). For technical convenience, we prefer to work with a lattice $\calD$ of dyadic cubes on $E_{0}$. As usual, we define
\begin{displaymath} \beta_{E_{0}}(Q) := \beta_{E_{0}}(B_{Q}), \qquad Q \in \calD. \end{displaymath}
Then, reducing "$M$" by a constant factor if necessary, it follows from \eqref{form3} that every $x \in B'$ is contained in $\geq M$ distinct cubes $Q \in \calD$ of side-length $0 < \ell(Q) \leq 1$ satisfying $\beta_{E_{0}}(Q) \geq \epsilon$. Moreover, since $B' \subset E \subset E_{0} \cap B(0,1)$, we may assume that $B_{Q} \subset B(0,C)$ for all the cubes $Q \in \calD$, for some $C \sim_{C_{0}} 1$. 

The main tool is that since $\Gamma$ is an $n$-dimensional $L$-Lipschitz graph in $\R^{d}$, it satisfies the WGL with constants depending only on $L$ and $d$. This follows from a more quantitative result -- a \emph{strong geometric lemma} for Lipschitz graphs -- of Dorronsoro \cite[Theorem 2]{MR796440} (or see \cite[Lemma 10.11]{DS1}). As a corollary of the WGL, the subset $\Gamma_{\mathrm{bad}}$ of points $x \in \Gamma \cap B(0,1)$ for which
\begin{equation}\label{form4} \beta_{\Gamma,\infty}(B(x,r)) \geq c\epsilon \end{equation}
for $\geq M/2$ distinct dyadic scales $0 < r < 1$ has measure $\calH^{n}(\Gamma_{\mathrm{bad}}) \ll 1$, and in particular $\calH^{n}(\Gamma_{\mathrm{bad}}) \leq \calH^{n}(B')/2$, assuming that $M \geq 1$ is large enough, depending only on $L,c,C_{0},d,H,\epsilon$, and $\theta$. In \eqref{form4}, $c > 0$ is a constant so small that
\begin{equation}\label{form93} B_{Q} \subset B(x,c^{-1}\ell(Q)/100) \quad \text{ for all } x \in Q. \end{equation}
In particular, $c$ only depends on the $n$-regularity constant of $E$. Further, in \eqref{form4}, the quantity $\beta_{\Gamma,\infty}(B(x,r))$ is the $L^{\infty}$-type $\beta$-number
\begin{displaymath} \beta_{\Gamma,\infty}(B(x,r)) = \inf_{V \in \mathcal{A}(d,n)} \sup_{y \in \Gamma \cap B(x,r)} \frac{\dist(y,V)}{r}. \end{displaymath}
As pointed out after Definition \ref{wgl}, the WGL holds for the $L^{\infty}$-type $\beta$-numbers if and only if it does for the $L^{1}$-type $\beta$-numbers $\beta_{\Gamma}(B(x,r))$ (Dorronsoro's strong geometric lemma holds for the latter, hence implies the WGL for the former).

We then focus attention on $B'' := B' \, \setminus \, \Gamma_{\mathrm{bad}} \subset \Gamma \cap B(0,1)$, which still satisfies
\begin{equation}\label{form85} \calH^{n}(B'') \geq \tfrac{1}{2}\calH^{n}(B') \sim_{\alpha,d,C_{0},\theta,H} 1, \end{equation}
recalling \eqref{form8}. Comparing \eqref{form3} and \eqref{form4}, we find that every point $x \in B''$ has the following property: there exist $M/2$ cubes $Q \in \calD$ such that $x \in Q$,
\begin{equation}\label{form5} \beta_{E_{0}}(Q) \geq \epsilon \quad \text{and} \quad \beta_{\Gamma,\infty}(B(x,c^{-1}\ell(Q)/100)) < c\epsilon. \end{equation} 

Consider now a cube $Q \in \calD$ containing at least one point $x \in B''$ such that \eqref{form5} holds. In particular, recalling the choice of $c > 0$ from \eqref{form93}, the intersection 
\begin{displaymath} \Gamma \cap B_{Q} \subset \Gamma \cap B(x,c^{-1}\ell(Q)/100) \end{displaymath}
is contained in a slab $T \subset \R^{d}$ (a neighbourhood of an $n$-plane) of width $\leq cc^{-1}\epsilon \ell(Q)/100 = \epsilon \ell(Q)/100$. Since $\beta_{E_{0}}(Q) \geq \epsilon$, however, we have
\begin{displaymath} \calH^{n}(\{y \in E_{0} \cap B_{Q} : y \notin 2T\}) \gtrsim \epsilon \calH^{n}(Q). \end{displaymath}
In other words, for every $Q \in \calD$ containing some $x \in B''$ such that \eqref{form5} holds, there exists a subset $E_{Q} \subset E_{0} \cap B_{Q} \subset B(0,C)$\begin{itemize}
\item of measure $\calH^{n}(E_{Q}) \gtrsim \epsilon \calH^{n}(Q)$ which is contained
\item in the $\sim \ell(Q)$-neighbourhood of $\Gamma$, yet
\item outside the $\sim \epsilon \ell(Q)$-neighbourhood of $\Gamma$.
\end{itemize} 
The collection of such cubes in $\mathcal{D}$ will be denoted $\mathcal{G}$. As observed above \eqref{form5}, we have
\begin{equation}\label{form6} \sum_{Q \in \mathcal{G}} \mathbf{1}_{Q}(x) \geq M/2, \qquad x \in B''. \end{equation}
On the other hand, the sets $E_{Q}$ have bounded overlap in the sense
\begin{equation}\label{form7} \sum_{Q \in \mathcal{G}} \mathbf{1}_{E_{Q}}(y) \lesssim_{\epsilon} 1, \qquad y \in \R^{d}, \end{equation}
since $y \in \R^{d}$ can only lie in the sets $E_{Q}$ associated to cubes $Q \in \calD$ with $\ell(Q) \sim_{\epsilon} \dist(y,\Gamma)$. Combining \eqref{form6}-\eqref{form7}, we find that
\begin{align*} 1 \gtrsim \calH^{n}(E_{0} \cap B(0,C)) & \geq \calH^{n}\left(\bigcup_{Q \in \mathcal{G}} E_{Q} \right)\\
& \gtrsim_{\epsilon} \sum_{Q \in \mathcal{G}} \calH^{n}(E_{Q}) \sim_{\epsilon} \sum_{Q \in \mathcal{G}} \calH^{n}(Q)\\
& \geq \int_{B''} \sum_{Q \in \mathcal{G}} \mathbf{1}_{Q}(x) \, d\calH^{n}(x) \gtrsim M\calH^{n}(B''). \end{align*} 
We have shown that $\calH^{n}(B'') \lesssim_{\epsilon} M^{-1}$. This inequality contradicts \eqref{form85} if $M \geq 1$ is large enough, depending on $\alpha,d,\epsilon,C_{0},\theta$, and $H$. The proof of Proposition \ref{prop1} is complete.  \end{proof}

%%%%%%%%%%%%%%%%%%%%

\section{Heavy trees have positive width}\label{s:heavyTrees}

We are equipped to prove Proposition \ref{treeProp}. Fix a heavy tree $\calT := \calT_{j}$, and recall from the heavy tree property \nref{T3} that if $Q \in \mathbf{Leaves}(\calT)$, then
\begin{displaymath} \card\{Q' \in \calT : Q \subset Q' \subset Q(\calT) \text{ and } \beta(Q') \geq \epsilon\} = M, \end{displaymath}
Moreover, by \nref{T1}-\nref{T2}, the total measure of $\mathbf{Leaves}(\calT)$ is 
\begin{equation}\label{form56} \mu(\cup \mathbf{Leaves}(\calT)) \geq \tfrac{1}{4}\mu(Q(\calT)). \end{equation}
Based on this information, we seek to verify the hypotheses of Proposition \ref{prop2}, which will eventually guarantee that $\w(\calT) \gtrsim 1$ and finish the proof of Proposition \ref{treeProp}. We split the argument into three parts. 
\subsection{Part I: Finding heavy cones} Abbreviate $Q_{0} := Q(\calT)$ and $\mathcal{L} := \mathbf{Leaves}(\calT)$. To avoid a rescaling argument later on, we assume with no loss of generality that
\begin{displaymath} \mu(Q_{0}) \sim \ell(Q_{0}) = 1. \end{displaymath}
For every $Q \in \mathcal{L}$, the PBP condition implies the existence of a plane $V_{Q} \in G(d,n)$ such that
\begin{equation}\label{form94} \calH^{n}(\pi_{V}(B_{Q} \cap E)) \geq \delta \mu(Q), \qquad V \in B(V_{Q},\delta). \end{equation} 
We would prefer that all the planes $V_{Q}$ are the same, and this can be arranged with little cost. Namely, pick a $\tfrac{\delta}{2}$-net $\{V_{1},\ldots,V_{m}\} \subset G(d,n)$ with $m \sim_{\delta,d,n} 1$, and note that for all $Q \in \mathcal{L}$, there is some $V_{j}$ such that $S_{j} := B(V_{j},\tfrac{\delta}{2}) \subset B(V_{Q},\delta) =: S_{Q}$. Therefore, by the pigeonhole principle, there is a fixed index $1 \leq j \leq m$ with the property
\begin{displaymath} \mathop{\sum_{Q \in \mathcal{L}}}_{S_{j} \subset S_{Q}} \mu(Q) \geq \frac{1}{m} \sum_{Q \in \mathcal{L}} \mu(Q) \stackrel{\eqref{form56}}{\sim}_{\delta,d,n} 1. \end{displaymath}
Let $\mathcal{L}_{G}$ be the \emph{good leaves} satisfying $S_{j} \subset S_{Q}$ for this $j$, and write $S := S_{j}$ and $V_{0} := V_{j}$. We have just argued that $\mu(\cup \mathcal{L}_{G}) \sim_{\delta,d,n} 1$, and \eqref{form94} holds for all $Q \in \mathcal{L}_{G}$, for all 
\begin{displaymath} V \in S = B_{G(d,n)}(V_{0},\tfrac{\delta}{2}). \end{displaymath}
From this point on, I cease recording the dependence of the "$\lesssim$" notation on the $n$-regularity and PBP constants $C_{0}$ and $\delta$.

For technical purposes, let us prune the set of good leaves a little further. Namely, apply the $5r$-covering theorem to the balls $10B_{Q}$, $Q \in \mathcal{L}_{G}$. As a result, we obtain a sub-collection of the good leaves, still denoted $\mathcal{L}_{G}$, with the separation property
\begin{equation}\label{form57} 10B_{Q} \cap 10B_{Q'} = \emptyset, \qquad Q,Q' \in \mathcal{L}_{G}, \, Q \neq Q', \end{equation} 
and such that the lower bound $\mu(\cup \mathcal{L}_{G}) \sim 1$ remains valid. 

Next we arrive at some geometric arguments. We may and will assume, with no loss of generality, and without further mention, that the radius of the ball $S = B_{G(d,n)}(V_{0},\tfrac{\delta}{2})$ is "small enough", in a manner depending only on $d$.

For every $Q \in \mathcal{L}_{G}$, pick an $n$-dimensional disc $D_{Q} \subset B_{Q}$ which is parallel to the plane $V_{0}$ and which satisfies
\begin{displaymath} \calH^{n}(D_{Q}) \sim \mu(Q) \text{ and } \calH^{n}(D_{Q} \cap E) = 0. \end{displaymath}
Such discs are pairwise disjoint by the separation property \eqref{form57}. We will also use frequently that the restrictions $\pi_{V}|_{D_{Q}} \colon D_{Q} \to V$ are bilipschitz for all $Q \in \mathcal{L}_{G}$ and $V \in S = B_{G(d,n)}(V_{0},\tfrac{\delta}{2})$ if $\delta > 0$ is small enough, as we assume. Therefore, the projections $\pi_{V}(D_{Q}) \subset V$ are $n$-regular ellipsoids which contain, and are contained in, $n$-dimensional balls of radius $\sim \mathrm{rad}(D_{Q})$.

We then consider the slightly augmented set $E_{+}$, where we have added the discs corresponding to all good leaves:
\begin{displaymath} E_{+} := E \cup \bigcup_{Q \in \calL_{G}} D_{Q} =: E \cup E_{D}. \end{displaymath}
The point behind the set $E_{D}$ can already be explained. Compare the two statements 
\begin{itemize}
\item[(a)] The Hardy-Littlewood maximal function of $\pi_{V\sharp}(\calH^{n}|_{E})$ is large at $x \in V \in S$, 
\item[(b)] The Hardy-Littlewood maximal function of $\pi_{V\sharp}(\calH^{n}|_{E_{D}})$ is large at $x \in V \in S$.
\end{itemize}
Statement (b) contains much more information! Statement (a) could e.g. be true because a single cube $Q \in \mathcal{L}_{G}$ satisfies $\pi_{V}(Q) = \{x\}$. But since $\pi_{V}|_{D_{Q}}$ is bilipschitz for all $Q \in \mathcal{L}_{G}$ and $V \in S$, statement (b) forces $\pi_{V}^{-1}\{x\}$ to intersect many distinct balls $B_{Q} \supset D_{Q}$. Recalling Proposition \ref{prop2}, this is helpful for finding a lower bound for $\w(\mathcal{T})$. 

Let us verify that $E_{+}$ is $n$-regular, with $n$-regularity constant $\lesssim 1$. We leave checking the lower bound to the reader. To check the upper bound, fix $x \in E_{+}$ and a radius $r > 0$. Since $E$ itself is $n$-regular, it suffices to show that
\begin{equation}\label{form59} \sum_{Q \in \calL_{G}} \calH^{n}(D_{Q} \cap B(x,r)) \lesssim r^{n}. \end{equation} 
Write
\begin{displaymath} \calL_{G}^{\leq} := \{Q \in \mathcal{L}_{G} : D_{Q} \cap B(x,r) \neq \emptyset \text{ and } \mathrm{rad}(D_{Q}) \leq r\} \end{displaymath}
and
\begin{displaymath} \calL_{G}^{>} := \{Q \in \mathcal{L}_{G} : D_{Q} \cap B(x,r) \neq \emptyset \text{ and } \mathrm{rad}(D_{Q}) > r\}. \end{displaymath}
For every $Q \in \calL^{\leq}_{G}$ we have $Q \subset B(x,C'r)$ for some constant $C' \sim 1$, so 
\begin{displaymath} \sum_{Q \in \calL_{G}^{\leq}} \calH^{n}(D_{Q} \cap B(x,r)) \lesssim \mathop{\sum_{Q \in \calL}}_{Q \subset B(x,C',r)} \mu(Q) \leq \mu(B(x,C'r)) \lesssim r^{n}. \end{displaymath}
Here we used that the leaves $\mathcal{L}$ consist of disjoint cubes. To finish the proof of \eqref{form59}, we claim that $\card \calL_{G}^{>} \leq 1$. Assume to the contrary that $D_{Q},D_{Q'} \in \calL_{G}^{>}$ with $Q \neq Q'$. Then certainly $2B_{Q} \cap B(x,r) \neq \emptyset \neq 2B_{Q'} \cap B(x,r)$, and both $B_{Q},B_{Q'}$ have diameters $\geq r$. This forces $10B_{Q} \cap 10B_{Q'} \neq \emptyset$, violating the separation condition \eqref{form57}. This completes the proof of \eqref{form59}.  

Let $\mu_{+} := \calH^{n}|_{E \cup E_{D}} = \mu + \sum_{Q \in \mathcal{L}_{G}} \calH^{n}|_{D_{Q}}$, and define the associated $\beta$-numbers
\begin{displaymath} \beta_{+}(B(x,r)) := \inf_{V \in \mathcal{A}(d,n)} \frac{1}{r^{n}} \int_{B(x,r)} \frac{\dist(y,V)}{r} \, d\mu_{+}(y), \qquad x \in E_{+}, \, r > 0.  \end{displaymath}
We next claim that for every $x \in E_{D}$ there exist $\gtrsim M$ distinct dyadic radii $0 < r \lesssim 1$ such that $\beta_{+}(B(x,r)) \gtrsim \epsilon$. This follows easily by recalling that if $x \in D_{Q}$ with $Q \in \calL_{Q} \subset \calL$, then 
\begin{displaymath} \card \{Q' \in \calT : Q \subset Q' \subset Q_{0} \text{ and } \beta(Q') \geq \epsilon\} = M \end{displaymath}
by the definition of good leaves, but let us be careful: let $x \in D_{Q}$, and let $Q' \in \calT$ be one of the ancestors of $Q$ with 
\begin{displaymath} \inf_{V \in \mathcal{A}(d,n)} \frac{1}{\mathrm{rad}(B_{Q'})^{n}} \int_{B_{Q'}} \frac{\dist(y,V)}{\mathrm{rad}(B_{Q'})} \, d\mu(y) = \beta(Q') \geq \epsilon. \end{displaymath}
Since $x \in D_{Q} \subset B_{Q} \subset B_{Q'}$, we have $B_{Q'} \subset B(x,r)$ for some (dyadic) $r \sim \mathrm{rad}(B_{Q'}) \lesssim 1$. Then, if $V \in \mathcal{A}(d,n)$ is arbitrary, we simply have
\begin{displaymath} \frac{1}{r^{n}} \int_{B(x,r)} \frac{\dist(y,V)}{r} \, d\mu_{+}(y) \geq \frac{1}{r^{n}} \int_{B(x,r)} \frac{\dist(y,V)}{r} \, d\mu(y) \gtrsim \epsilon, \end{displaymath}
which proves that $\beta_{+}(B(x,r)) \gtrsim \epsilon$. A fixed radius "$r$" can only be associated to $\lesssim 1$ cubes $Q'$ in the ancestry of $Q$, so $\gtrsim M$ of them arise in the manner above. The claim follows.

We note that 
\begin{equation}\label{form60} \mu_{+}(E_{D}) \gtrsim \mu(\cup \mathcal{L}_{G}) \sim 1. \end{equation}
We aim to apply Proposition \ref{prop1} to the set $E_{D}$, but we will perform a final pruning before doing so. Let $c > 0$ be a small constant to be determined soon, and let $\calL_{G,\mathrm{light}} \subset \calL_{G}$ consist of the good leaves with the following property: there exists a point $x_{Q} \in D_{Q}$ and a radius $0 < r_{Q} \leq 1$ such that
\begin{equation}\label{form61} \mu_{+}(E_{D} \cap B(x_{Q},r_{Q})) \leq c r^{n}_{Q}. \end{equation} 
Evidently $D_{Q} \subset B(x_{Q},r_{Q}/5)$ if $c > 0$ is small enough, since if $D_{Q} \not\subset B(x_{Q},r_{Q}/5)$, then 
\begin{displaymath} \mu_{+}(E_{D} \cap B(x,r_{Q})) \geq \mu_{+}(D_{Q} \cap B(x,r_{Q}/5)) \sim r^{n}_{Q}. \end{displaymath}
We also observe that since $x_{Q} \in D_{Q} \subset B_{Q} \subset B_{Q_{0}}$, and $r_{Q} \leq 1 = \ell(Q_{0})$, we have $B(x_{Q},r_{Q}) \subset 2B_{Q_{0}}$ for all $Q \in \calL_{G,\mathrm{light}}$. Now, use the $5r$ covering theorem to find a subset $\calL' \subset \mathcal{L}_{G,\mathrm{light}}$ such that the associated balls $B(x_{Q},r_{Q}/5)$ are disjoint, and 
\begin{displaymath} \bigcup_{Q \in \calL_{G,\mathrm{light}}} D_{Q} \subset \bigcup_{Q \in \mathcal{L}_{G,\mathrm{light}}} B(x_{Q},\tfrac{1}{5}r_{Q}) \subset \bigcup_{Q \in \mathcal{L}'} B(x_{Q},r_{Q}). \end{displaymath}
It follows from \eqref{form61}, and the $n$-regularity of $\mu_{+}$, that
\begin{displaymath} \mu_{+}\Big( \bigcup_{Q \in \mathcal{L}_{G,\mathrm{light}}} D_{Q} \Big) \leq c \sum_{Q \in \mathcal{L}'} r^{n}_{Q} \lesssim c \sum_{Q \in \mathcal{L}'} \mu_{+}(B(x_{Q},\tfrac{r_{Q}}{5})) \leq c\mu_{+}(2B_{Q_{0}}) \lesssim c. \end{displaymath} 
Comparing this upper bound with \eqref{form60}, we find that if $c > 0$ was chosen small enough, depending only on the PBP and $n$-regularity constants of $E$, then 
\begin{displaymath} \sum_{Q \in \mathcal{L}_{G,\mathrm{heavy}}} \mu_{+}(D_{Q}) \gtrsim 1,  \end{displaymath} 
where $\mathcal{L}_{G,\mathrm{heavy}} = \mathcal{L}_{G} \, \setminus \, \mathcal{L}_{G,\mathrm{light}}$. Let $E_{D,\mathrm{dense}}$ be the union of the discs $D_{Q}$ with $Q \in \mathcal{L}_{G,\mathrm{heavy}}$. We summarise the properties of $E_{D,\mathrm{dense}} \subset E_{D} \subset E_{+}$:
\begin{enumerate}
\item\label{1} $\mu_{+}(E_{D,\mathrm{dense}}) \sim 1$,
\item\label{2} If $x \in E_{D,\mathrm{dense}}$, there are $\gtrsim M$ dyadic scales $0 < r \lesssim 1$ such that $\beta_{+}(B(x,r)) \gtrsim \epsilon$,
\item\label{3} If $x \in E_{D,\mathrm{dense}}$, then $\mu_{+}(E_{D} \cap B(x,r)) \gtrsim r$ for all $0 < r \leq 1$.
\end{enumerate}
We then apply Proposition \ref{prop1} to the set $E_{D,\mathrm{dense}}$ with a "multiplicity" parameter $H \geq 1$ to be chosen later. As usual, the choice of the parameter $H$ will eventually only depend on the $n$-regularity and PBP constants of $E$. The parameters $\alpha$ and $\theta$ in the statement of the proposition are set to be such that $\alpha \sim_{d,\delta} 1$ (specifics to follow later), and $\theta \sim 1$ is so small that $\calH^{n}(E_{D,\mathrm{dense}}) \geq \theta$, which is possible by (1) above. As a good first approximation of how to choose $\alpha$, recall from Lemma \ref{MOLemma} that if $x \in \R^{d}$ and $|\pi_{V_{0}}(x)| \leq \alpha |x|$, where $\alpha = \alpha(d,\delta) > 0$ is small enough, then there exists a plane $V \in B_{G(d,n)}(V_{0},\tfrac{\delta}{2}) = S$ such that $\pi_{V}(x) = 0$. In symbols, the previous statement is equivalent to
\begin{equation}\label{form96} X(0,V_{0},\alpha) \subset \bigcup_{V \in S} V^{\perp} =: \mathcal{C}(S). \end{equation}
In fact, in the case $n = d - 1$, this would be a suitable definition for $\alpha$, and the reader may think that $\alpha$ is at least so small that \eqref{form96} holds. In the case $n < d - 1$, additional technicalities force us to pick $\alpha$ slightly smaller.

Proposition \ref{prop1} then states that if $M \geq 1$ is chosen large enough, in a manner depending only on $\alpha,H,d,\delta,\epsilon,\theta$, and the $n$-regularity constant of $E$, the following holds: there exists a subset $G \subset E_{D,\mathrm{dense}}$ of measure 
\begin{equation}\label{form73} 1 \gtrsim \calH^{n}(G) \gtrsim \theta \sim 1 \end{equation}
with the property
\begin{equation}\label{form62} \card \{j \geq 0 : X(x,V_{0},\tfrac{\alpha}{2},2^{-j - 1},2^{-j}) \cap E_{D,\mathrm{dense}} \neq \emptyset\} \geq H, \qquad x \in G. \end{equation}
(The upper bound in \eqref{form73} follows from $G \subset E_{D}$ and $\diam(E_{D}) \lesssim \ell(Q_{0}) = 1$). We next upgrade \eqref{form62} to a measure estimate, using the definition of $E_{D,\mathrm{dense}}$. Namely, recall from \eqref{3} above that if $y \in E_{D,\mathrm{dense}}$, then $\mu_{+}(E_{D} \cap B(y,r)) \gtrsim r^{n}$ for all $0 < r \leq 1$. By definitions and a few applications of the triangle inequality,
\begin{displaymath} y \in X(x,V_{0},\tfrac{\alpha}{2},2^{-j - 1},2^{-j}) \quad \Longrightarrow \quad B(y,\alpha2^{-j - 10}) \subset X(x,V_{0},\alpha,2^{-j - 2},2^{-j + 1}), \end{displaymath}
and hence
\begin{equation}\label{form95} \calH^{n}(E_{D} \cap X(x,V_{0},\alpha,2^{-j - 2},2^{-j + 1})) \gtrsim \mu_{+}(E_{D} \cap B(y,\alpha 2^{-j - 10})) \gtrsim 2^{-jn} \end{equation} 
for all those scales $2^{-j}$ such that $X(x,V_{0},\tfrac{\alpha}{2},2^{-j - 1},2^{-j})$ contains some $y \in E_{D,\mathrm{dense}}$. (Here we used that $\alpha \sim_{d,\delta} 1$.) For $x \in G$, the number of such scales "$2^{-j}$" is no smaller than $H$, by \eqref{form62}, for every such "$2^{-j}$", it follows from \eqref{form95} that at least one of the three scales $2^{-i} \in \{2^{-j - 1},2^{-j},2^{-j + 1}\}$ satisfies $\calH^{n}(E_{D} \cap X(x,V_{0},\alpha,2^{-i - 1},2^{-i})) \geq c2^{-in}$. Here $c \sim 1$ is a constant which records for the implicit constants in \eqref{form95}. Therefore, replacing "$H$" by "$H/3$" without altering notation, we have just proven the following:
\begin{equation}\label{form63} \card \{j \geq 0 : \calH^{n}(E_{D} \cap X(x,V_{0},\alpha,2^{-j - 1},2^{-j})) \geq c 2^{-jn}\} \geq H, \quad x \in G.  \end{equation}
\subsection{Part II: Besicovitch-Federer argument} By following the classical argument of Besicovitch and Federer, we aim to use \eqref{form63} to show that the projections of $E_{D}$ to planes close to $V_{0}$ have plenty of of overlap. This part of the argument will be quite familiar to readers acquainted with the proof of the Besicovitch-Federer projection theorem.

For $V \in S = B_{G(d,n)}(V_{0},\tfrac{\delta}{2})$, write
\begin{displaymath} f_{V} := \sum_{Q \in \mathcal{L}_{G}} \mathbf{1}_{\pi_{V}(B_{Q})}, \end{displaymath}
interpreted as a function on $\R^{n}$, and let $\mathcal{M} f_{V}$ stand for the centred Hardy-Littlewood maximal function of $f_{V}$. We will prove the following claim:
\begin{claim}\label{c1} For every $x \in G$, there exists a subset $S_{x} \subset S$ of measure $\gamma_{d,n}(S_{x}) \gtrsim 1/\sqrt{H}$ with the following property:
\begin{equation}\label{form67} \mathcal{M}f_{V}(\pi_{V}(x)) \gtrsim \sqrt{H}, \qquad V \in S_{x}. \end{equation} 
\end{claim}

As usual, the implicit constants here are allowed to depend on $d$, and the $n$-regularity and PBP constants of $E$. During the proof of the claim, we use the abbreviation
\begin{equation}\label{form112} E_{j,x} := E_{D} \cap X(x,V_{0},\alpha,2^{-j - 1},2^{-j}), \qquad j \geq 0. \end{equation}
By \eqref{form63}, there exist $H$ distinct indices $j \geq 0$ such that $\calH^{n}(E_{j,x}) \geq c2^{-jn}$. The proof of the claim splits into two cases: either there is at least one of these indices "$j$" such that $E_{j,x}$ meets only a few planes $\pi_{V}^{-1}\{\pi_{V}(x)\}$, $V \in S$, or then $E_{j,x}$ meets fairly many of the planes $\pi_{V}^{-1}\{\pi_{V}(x)\}$, $V \in S$, for every one of the $H$ indices "$j$".

\medskip

\noindent\textbf{Case 1.} Fix $x \in G$, assume with no loss of generality that $x = 0$. This has the notational benefit that $\pi_{V}^{-1}\{\pi_{V}(x)\} = V^{\perp}$ for $V \in G(d,n)$. Assume that there exists at least one index $j \geq 0$ such that $\calH^{n}(E_{j,x}) \geq c2^{-jn}$, and
\begin{equation}\label{form64} \gamma_{d,n}(\{V \in S : V^{\perp} \cap E_{j,0} \neq \emptyset\}) \leq \frac{1}{\sqrt{H}}. \end{equation} 
Fix such an index $j \geq 0$, and abbreviate $E_{j,0} := E_{0}$. Then \eqref{form64} will imply that most of the (non-negligible) $\calH^{n}$ mass of $E_{0} \subset X(0,V_{0},\alpha)$ is contained in narrow slabs around $(d - n)$-planes with "high density". As in the classical proof of the Besicovitch-Federer projection theorem, the case $n < d - 1$ requires integralgeometric considerations, whose necessity will only become clear at the very end of \textbf{Case 1}. Fortunately, they also make technical sense in the case $n = d - 1$ (they just become trivial), so the case $n = d - 1$ does not require separate treatment. As in Section \ref{s:Grassmannian}, we define
\begin{displaymath} G(W,n) := \{V \in G(d,n) : V \subset W\} \cong G(n + 1,n), \qquad W \in G(d,n + 1), \end{displaymath}
and we write $\gamma_{W,n + 1,n}$ for the $\mathcal{O}(d)$-invariant probability measure on $G(W,n)$. The metric on $G(W,n)$ is inherited from $G(d,n)$. Recall the Fubini formula established in Lemma \ref{l:fubini}:

\begin{equation}\label{fubini} \gamma_{d,n}(B) = \int_{G(d,n + 1)} \gamma_{W, n + 1,n}(B) \, d\gamma_{d,n + 1}(W) \end{equation}
for $B \subset G(d,n)$ Borel. We will need to find a Borel set $\mathcal{W} \subset G(d,n + 1)$, in fact a ball, which may depend on $j$ and $x$, with the following properties:
\begin{enumerate}
\item[(W1) \phantomsection\label{W1}] $\gamma_{d,n + 1}(\mathcal{W}) \sim_{d,\delta} 1$,
\item[(W2) \phantomsection\label{W2}] For every $W \in \mathcal{W}$, the set $S \cap G(W,n)$ contains a ball $S_{W} = B_{G(W,n)}(V_{W},\tfrac{\delta}{4})$,
\item[(W3) \phantomsection\label{W3}] There exists a subset $E_{\mathcal{W},0} \subset E_{0}$ of measure $\calH^{n}(E_{\mathcal{W},0}) \geq c2^{-jn}$ with the property
\begin{displaymath} E_{\mathcal{W},0} \subset \bigcup_{V \in S_{W}} V^{\perp}, \qquad W \in \mathcal{W}. \end{displaymath} 
\end{enumerate}
The "$c$" appearing in property \nref{W3} may be a constant multiple (depending on $\delta,d$) of the constant in $\calH^{n}(E_{0}) \geq c2^{-jn}$. Finding $\mathcal{W}$ with the properties \nref{W1}-\nref{W3} is easy if $n = d - 1$, so let us discuss this case first to get some intuition. Simply take $\mathcal{W} := G(d,d) = \{\R^{d}\}$. Note that in this case $G(W,n) \equiv G(d,n)$. Evidently \nref{W1}-\nref{W2} are satisfied, even with $S_{W} := S$. Also, \nref{W3} is satisfied with $E_{\mathcal{W},0} := E_{0}$ by \eqref{form96}, which implies that $E_{0} \subset X(0,V_{0},\alpha) \subset \bigcup_{V \in S} V^{\perp}$.

We then treat the general case. In the process, we also finally fix the angular parameter $\alpha \sim_{d,\delta} 1$. Recall that $E_{0} \subset X(0,V_{0},\alpha,2^{-j - 1},2^{-j})$, that is, $|\pi_{V_{0}}(z)| \leq \alpha |z|$ and $|z| \sim 2^{-j}$ for all $z \in E_{0}$. Start by choosing a point $z_{0} \in E_{0}$ such that
\begin{equation}\label{form100} \calH^{n}(E_{0} \cap B(z_{0},\rho 2^{-j})) \gtrsim_{\delta,d} 2^{-jn}, \end{equation}
where $0 < \rho \leq \min\{\tfrac{1}{10},\alpha,\delta\}$ is a parameter to be chosen momentarily (we will have $\rho \sim_{\delta,d} 1$). We then define 
\begin{displaymath} E_{\mathcal{W},0} := E_{0} \cap B(z_{0},\rho 2^{-j}), \end{displaymath}
so at least the measure estimate in \nref{W3} is satisfied by \eqref{form100}. Write $W_{0} := \spa(V_{0},z_{0}) \in G(d,n + 1)$ (evidently $z_{0} \notin V_{0}$ since $|\pi_{V_{0}}(z_{0})| < |z_{0}|$), and set $\mathcal{W} := B(W_{0},\rho)$. Then $\gamma_{d,n + 1}(\mathcal{W}) \sim_{d,\delta} 1$, so property \nref{W1} is satisfied. 

We next verify  \nref{W2}. Let $W \in \mathcal{W}$, that is, $d(W,W_{0}) \leq \rho$. Then, since $V_{0} \subset W_{0}$, Lemma \ref{GrLemma} implies that there exists a plane $V_{W} \in G(W,n)$ with $d(V_{W},V_{0}) \lesssim \rho$. In particular, $V_{W} \in B_{G(d,n)}(V_{0},\tfrac{\delta}{4})$ if $\rho$ is chosen small enough, and consequently
\begin{displaymath} S_{W} := B_{G(W,n)}(V_{W},\tfrac{\delta}{4}) \subset S. \end{displaymath}
This completes the proof of \nref{W2}.

To prove \nref{W3}, we need to check that if $W \in \mathcal{W}$ and $z \in E_{\mathcal{W},0}$, then there exists a plane $V \in S_{W}$ with $\pi_{V}(z) = 0$. This will be accomplished by an application of Lemma \ref{MOLemma} inside $W \cong \R^{n + 1}$. First, since $z \in E_{\mathcal{W},0} \subset E_{0}$, $V_{W} \subset W$, and $d(V_{W},V_{0}) \lesssim \rho \leq \alpha$, we have
\begin{equation}\label{form106} |\pi_{V_{W}}(\pi_{W}(z))| = |\pi_{V_{W}}(z)| \leq d(V_{W},V_{0}) \cdot |z| + |\pi_{V_{0}}(z)| \lesssim \alpha |z|. \end{equation}
Second,
\begin{equation}\label{form107} |\pi_{W}(z)| \geq |\pi_{W_{0}}(z_{0})| - d(W,W_{0}) \cdot |z_{0}| - |z - z_{0}| \gtrsim |z|, \end{equation}
using that $z_{0} \in W_{0}$, and $z \in B(z_{0},\rho 2^{-j}) \subset B(z_{0},|z_{0}|/2)$, and $d(W,W_{0}) \leq \rho$. Combining \eqref{form106}-\eqref{form107}, and setting $z_{W} := \pi_{W}(z) \in W$, we find that
\begin{equation}\label{form108} |\pi_{V_{W}}(z_{W})| \lesssim \alpha |z_{W}|. \end{equation}
Finally, the estimate \eqref{form108} allows us to apply Lemma \ref{MOLemma} to the point $z_{W} \in W$ in the space $G(W,n) \cong G(n + 1,n)$. The conclusion is that if $\alpha$ is small enough, depending only on $\delta,n$, then there exists a plane $V \in B_{G(W,n)}(V_{W},\tfrac{\delta}{4}) = S_{W}$ such that $\pi_{V}(z_{W}) = 0$. But now $V \subset W$, and $\pi_{W}(z - z_{W}) = 0$, so also $\pi_{V}(z) = \pi_{V}(z_{W}) + \pi_{V}(z - z_{W}) = 0$. This is what we claimed, so the proof of \nref{W3} is complete.

After the preparations \nref{W1}-\nref{W3}, we can get to the business of verifying Claim \ref{c1} in \textbf{Case 1}. Recall from the main assumption \eqref{form64} that $\gamma_{d,n}(\{V \in S : V^{\perp} \cap E_{0} \neq \emptyset\}) \leq 1/\sqrt{H}$. Combined with the Fubini formula \eqref{fubini}, this implies that the set of planes $W \in G(d,n + 1)$ such that
\begin{equation}\label{form104} \gamma_{W,n + 1,n}(\{V \in S_{W} : V^{\perp} \cap E_{0} \neq \emptyset\}) \geq \frac{C}{\sqrt{H}} \end{equation} 
has $\gamma_{d,n + 1}$-measure at most $C^{-1}$, for $C \geq 1$. Choose $C \sim_{\delta} 1$ here so large that the planes $W \in G(d,n + 1)$ in question have total measure $\leq \tfrac{1}{2}\gamma_{d,n + 1}(\mathcal{W})$. After discarding these "bad" planes from $\mathcal{W}$, we may assume that the opposite of \eqref{form104} holds for all $W \in \mathcal{W}$:
\begin{equation}\label{form105} \gamma_{W,n + 1,n}(\{V \in S_{W} : V^{\perp} \cap E_{0} \neq \emptyset\}) \leq \frac{C}{\sqrt{H}}. \end{equation} 
Fix $W \in \mathcal{W}$, so \eqref{form105} holds, and abbreviate $\gamma_{W,n + 1,n} =: \gamma_{n + 1,n}$. Then, let $\mathcal{S}$ be a system of dyadic cubes on the ($n$-regular) ball $S_{W} \subset G(W,n)$, with top cube $S_{W}$. Then, cover the set
\begin{displaymath} \bar{S}_{W} := \{V \in S_{W} : V^{\perp} \cap E_{0} \neq \emptyset\} \end{displaymath}
by a disjoint collection $\mathcal{Q} \subset \mathcal{S}$ of these cubes such that
\begin{displaymath} \sum_{Q \in \mathcal{Q}} \gamma_{n + 1,n}(Q) \leq \frac{2C}{\sqrt{H}}. \end{displaymath}
For $Q \in \mathcal{Q}$, write $\mathcal{C}(Q) := \cup \{V^{\perp} : V \in Q\}$, generalising the notation $\mathcal{C}(S)$ introduced in \eqref{form96}. Since $\bar{S}_{W}$ is covered by the cubes $Q \in \mathcal{Q}$, the set $E_{\mathcal{W},0} \subset E_{0} \cap \bigcup_{V \in S_{W}} V^{\perp}$ is covered by the cones $\mathcal{C}(Q)$, $Q \in \mathcal{Q}$. Now, let $\mathcal{Q}_{\mathrm{light}}$ be the cubes $Q \in \mathcal{Q}$ satisfying
\begin{equation}\label{form66} \calH^{n}(\mathcal{C}(Q) \cap E_{\mathcal{W},0}) \leq \tfrac{c}{4C}\sqrt{H} \cdot 2^{-jn} \cdot \gamma_{n + 1,n}(Q). \end{equation} 
Then,
\begin{displaymath} \sum_{Q \in \mathcal{Q}_{\mathrm{light}}} \calH^{n}(\mathcal{C}(Q) \cap E_{\mathcal{W},0}) \leq \tfrac{c}{4C}\sqrt{H} \cdot 2^{-jn} \sum_{Q \in \mathcal{Q}} \gamma_{n + 1,n}(Q) \leq \tfrac{c}{2} \cdot 2^{-jn}. \end{displaymath}
Recalling from \nref{W3} that $\calH^{n}(E_{\mathcal{W},0}) \geq c2^{-jn}$, and that $E_{\mathcal{W},0}$ is covered by the union of the cones $\mathcal{C}(Q)$, $Q \in \mathcal{Q}$, we infer that there is a subset $\bar{E}_{\mathcal{W},0} \subset E_{\mathcal{W}}$ of measure $\mathcal{H}^{n}(\bar{E}_{\mathcal{W},0}) \geq \tfrac{c}{2} \cdot 2^{-jn}$ which is covered by the union of the cones $\mathcal{C}(Q)$, $Q \in \mathcal{Q} \, \setminus \, \mathcal{Q}_{\mathrm{light}}$. Every cube $Q \in \mathcal{Q} \, \setminus \, \mathcal{Q}_{\mathrm{light}}$ satisfies the inequality reverse to \eqref{form66}, and is consequently contained in some maximal cube in $\mathcal{S}$ with this property. Let $\mathcal{Q}_{\mathrm{heavy}}$ be the collection of such maximal (hence disjoint) cubes. Then, since $Q \subset Q'$ implies $\mathcal{C}(Q) \subset \mathcal{C}(Q')$, we see that $\bar{E}_{\mathcal{W},0}$ is also covered by the union of the cones $\mathcal{C}(Q)$, $Q \in \mathcal{Q}_{\mathrm{heavy}}$, and consequently
\begin{equation}\label{form68} \sum_{Q \in \mathcal{Q}_{\mathrm{heavy}}} \calH^{n}(\mathcal{C}(Q) \cap \bar{E}_{\mathcal{W},0}) \geq \tfrac{c}{2} \cdot 2^{-jn}. \end{equation}
We moreover claim that the union of the heavy cubes, denoted $H_{W}$, satisfies
\begin{equation}\label{form65} \gamma_{n + 1,n}(H_{W}) = \sum_{Q \in \mathcal{Q}_{\mathrm{heavy}}} \gamma_{n + 1,n}(Q) \gtrsim \frac{1}{\sqrt{H}}.  \end{equation} 
Indeed, if $S_{W} \in \mathcal{Q}_{\mathrm{heavy}}$, there is nothing to prove, since $\gamma_{n,n + 1}(S_{W}) \sim_{\delta,d} 1$. If, on the other hand, $S_{W} \notin \mathcal{Q}_{\mathrm{heavy}}$, then the parent $\hat{Q}$ of every cube $Q \in \mathcal{Q}_{\mathrm{heavy}}$ satisfies \eqref{form66}, by the maximality of $Q$. Of course \eqref{form68} remains valid if we replace "$Q$" by "$\hat{Q}$". Putting these pieces together, we find that
\begin{align*} \sum_{Q \in \mathcal{Q}_{\mathrm{heavy}}} \gamma_{n + 1,n}(Q) & \gtrsim \sum_{Q \in \mathcal{Q}_{\mathrm{heavy}}} \gamma_{n + 1,n}(\hat{Q})\\
& \geq \frac{C \cdot 2^{jn + 2}}{c\sqrt{H}}\sum_{Q \in \mathcal{Q}_{\mathrm{heavy}}} \calH^{n}(\mathcal{C}(\hat{Q}) \cap \bar{E}_{\mathcal{W},0}) \stackrel{\eqref{form68}}{\geq} \frac{1}{\sqrt{H}}. \end{align*}
This completes the proof of \eqref{form65}.

We are now ready to prove Claim \ref{c1} in \textbf{Case 1}, that is, define the set $S_{x} = S_{0} \subset S$ such that \eqref{form67} holds for all $V \in S_{0}$. Define
\begin{equation}\label{form102} S_{0} := \bigcup_{W \in \mathcal{W}} H_{W} \subset \bigcup_{W \in \mathcal{W}} S_{W} \subset S. \end{equation}
Then, by the Fubini formula \eqref{fubini}, and the uniform lower bound \eqref{form65}, we have
\begin{displaymath} \gamma_{d,n}(S_{0}) \geq \int_{\mathcal{W}} \gamma_{W,n + 1,n}(H_{W}) \, d\gamma_{d,n + 1}(W) \stackrel{\eqref{form65}}{\gtrsim} \frac{\gamma_{d,n + 1}(\mathcal{W})}{\sqrt{H}} \stackrel{\nref{W1}}{\sim_{\delta,d}} \frac{1}{\sqrt{H}}, \end{displaymath}
as required by Claim \ref{c1}. It remains to establish the lower bound \eqref{form67}, namely that if $V \in S_{0} (= S_{x})$, then $\mathcal{M}f_{V}(\pi_{V}(x)) = \mathcal{M}f_{V}(0) \gtrsim \sqrt{H}$. Fix $V \in S_{0}$, let first $W \in \mathcal{W}$ be such that $V \in H_{W}$, and then let $Q \in \mathcal{Q}_{W,\mathrm{heavy}} = \mathcal{Q}_{\mathrm{heavy}}$ be the unique cube with $V \in Q$ (we do not claim, however, that the choice of $W$ would be unique). By definitions, especially recalling that $E_{\mathcal{W},0} \subset E_{0} \subset E_{D} \cap \bar{B}(2^{-j}) \, \setminus B(2^{-j - 1})$, we have
\begin{equation}\label{form70} \calH^{n}(\mathcal{C}(Q,2^{-j - 1},2^{-j}) \cap E_{D}) \geq \calH^{n}(\mathcal{C}(Q) \cap E_{\mathcal{W},0}) \geq \tfrac{c}{4C}\sqrt{H} \cdot 2^{-jn} \cdot \gamma_{n + 1,n}(Q), \end{equation}
where of course $\mathcal{C}(Q,r,R) := \mathcal{C}(Q) \cap \bar{B}(R) \, \setminus \, B(r)$, and we recall that $\mathcal{C}(Q) = \{V^{\perp} : V \in Q\}$. Note that $\mathcal{C}(Q,2^{-j - 1},2^{-j}) \subset T = T_{V}$, where $T \subset \R^{d}$ is a slab of the form
\begin{displaymath} T := \pi_{V}^{-1}[B(0,C2^{-j}\ell(Q))] \end{displaymath}
of width $\sim_{d} 2^{-j}\ell(Q)$ around the plane $V^{\perp} \in G(d,d - n)$. Indeed, if $x \in \mathcal{C}(Q,2^{-j - 1},2^{-j})$, then $\pi_{V'}(x) = 0$ for some $V' \in Q$. Then $d(V,V') \lesssim_{d} \ell(Q)$, and $|\pi_{V}(x)| \leq d(V,V') \cdot |x| \lesssim 2^{-j}\ell(Q)$, which means that $x \in T$ if the constant $C \geq 1$ is chosen appropriately.

Write $B_{V} := B(0,C2^{-j}\ell(Q)) \subset V$. With this notation, recalling that $D_{Q} \subset B_{Q}$, and using that the projections $\pi_{V}|_{D_{Q}} \colon D_{Q} \to V$ are bilipschitz for $Q \in \mathcal{L}_{G}$ and $V \in S_{0} \subset S$, we infer that 
\begin{align*} \mathcal{M}f_{V}(0) & \geq \frac{1}{\mathrm{rad}(B_{V})^{n}} \int_{B_{V}} \sum_{Q \in \mathcal{L}_{G}} \mathbf{1}_{\pi_{V}(D_{Q})}(y) \, dy\\
& = \frac{1}{\mathrm{rad}(B_{V})^{n}} \sum_{Q \in \mathcal{L}_{G}} \calH^{n}(B_{V} \cap \pi_{V}(D_{Q}))\\
& \sim \frac{1}{\mathrm{rad}(B_{V})^{n}} \sum_{Q \in \mathcal{L}_{G}} \calH^{n}(T \cap D_{Q})\\
& \geq \frac{\calH^{n}(T \cap E_{D})}{\mathrm{rad}(B_{V})^{n}} \stackrel{\eqref{form70}}{\geq} \frac{(c/4)\sqrt{H} \cdot 2^{-jn} \cdot \gamma_{n + 1,n}(Q)}{\mathrm{rad}(B_{V})^{n}} \sim \sqrt{H}. \end{align*} 
In final estimate, we used that $\gamma_{n + 1,n}(Q) \sim \ell(Q)^{n}$. This is the whole point of the integralgeometric argument: without splitting $G(d,n)$ into a "product" of $G(d,n + 1)$ and $G(W,n)$, we could have, more easily, reached the penultimate estimate with "$\gamma_{d,n}(Q)$" in place of "$\gamma_{n + 1,n}(Q)$". But $\gamma_{d,n}(Q) \sim \ell(Q)^{n(d - n)} \ll \ell(Q)^{n}$ if $n < d - 1$, and the final estimate would have failed. We have now proved Claim \ref{c1} in \textbf{Case 1}.

\medskip

\noindent \textbf{Case 2.} Again, fix $x \in G$, assume with no loss of generality that $x = 0$, and let $j_{1},\ldots,j_{H} \geq 0$ be distinct scale indices such that $\calH^{n}(E_{j_{i},0}) \geq c2^{-j_{i}n}$ for all $1 \leq i \leq H$, recall the notation from \eqref{form112}. This time, we assume that
\begin{equation}\label{form71} \gamma_{d,n}(\bar{S}_{0,i}) \geq \frac{1}{\sqrt{H}}, \qquad 1 \leq i \leq H, \end{equation}
where $\bar{S}_{0,i} := \{V \in S : V^{\perp} \cap E_{j_{i},0} \neq \emptyset\}$. It follows from \eqref{form71} that
\begin{equation}\label{form72} \int_{S} \sum_{i = 1}^{H} \mathbf{1}_{\bar{S}_{0,i}}(V) \, \gamma_{d,n}(V) \geq \sqrt{H}. \end{equation}
Let
\begin{displaymath} S_{x} := S_{0} := \left\{V \in S : \sum_{i = 1}^{H} \mathbf{1}_{\bar{S}_{0,i}}(V) \geq \sqrt{H} \right\}, \end{displaymath} 
Then, it follows by splitting the integration in \eqref{form72} to $S \, \setminus \, S_{0}$ and $S_{0}$, that
\begin{displaymath} \sqrt{H} \leq \sqrt{H} \cdot \gamma_{d,n}(S \, \setminus \, S_{0}) + H \cdot \gamma_{d,n}(S_{0}). \end{displaymath} 
Recalling that $\gamma_{d,n}(S) \leq \tfrac{1}{2}$ (that is, $S = B_{G(d,n)}(V_{0},\tfrac{\delta}{2})$ is a fairly small ball), we find that $\gamma_{d,n}(S_{0}) \gtrsim 1/\sqrt{H}$, as required by Claim \ref{c1}. It remains to check that $\mathcal{M}f_{V}(\pi_{V}(x)) = \mathcal{M}f_{V}(0) \gtrsim \sqrt{H}$ whenever $V \in S_{0}$.

Fixing $V \in S_{0}$, it follows by definition that there are $\geq \sqrt{H}$ indices $i \in \{1,\ldots,H\}$ with the property that $V \in \bar{S}_{0,i}$, which meant by definition that
\begin{displaymath} V^{\perp} \cap E_{D} \supset V^{\perp} \cap E_{j_{i},0} \neq \emptyset. \end{displaymath} 
For each of these indices $i$, the plane $V^{\perp}$ intersects at least one of the discs $D_{Q}$ with $Q \in \mathcal{L}_{G}$, whose union is $E_{D}$. Moreover, since the sets $E_{j,0} \subset \bar{B}(2^{-j}) \, \setminus \, \bar{B}(2^{-j - 1})$ are disjoint for distinct indices $j \geq 0$, we conclude that $V^{\perp}$ meets $\geq \sqrt{H}$ distinct discs $D_{Q}$. Consequently, recalling also that $D_{Q} \subset B_{Q}$ for all $Q \in \mathcal{L}_{G}$,
\begin{displaymath} f_{V}(0) = \sum_{Q \in \mathcal{L}_{G}} \mathbf{1}_{\pi_{V}(B_{Q})}(0) \geq \card\{Q \in \mathcal{L}_{G} : V^{\perp} \cap D_{Q} \neq \emptyset\} \geq \sqrt{H}. \end{displaymath} 
A similar lower bound for $\mathcal{M}f_{V}$ follows easily from the special structure of $f_{V}$: whenever $V \in S_{0} \subset S$ and $f_{V}(0) \geq \sqrt{H}$, we may pick the $h := \sqrt{H}$ largest balls $B_{1},\ldots,B_{h}$ of the form $\pi_{V}(B_{Q}) \subset V$, $Q \in \mathcal{L}_{G}$, which contain $0$. Writing $r := \min\{\mathrm{rad}(B_{k}) : 1 \leq k \leq h\}$,
\begin{displaymath} \mathcal{M}f_{V}(0) \geq \frac{1}{r^{n}}\int_{B_{V}(r)} f_{V}(y) \, d\calH^{n}(y) \geq \frac{1}{r^{n}} \sum_{k = 1}^{h} \calH^{n}(B_{V}(r) \cap B_{k}) \gtrsim \sqrt{H}, \end{displaymath} 
as claimed. This completes the proof of \eqref{form67}, and Claim \ref{c1}, in \textbf{Case 2}.

\subsection{Part III: Conclusion} We then proceed with the proof of Proposition \ref{treeProp}. Recall from \eqref{form73} that $\calH^{n}(G) \sim 1$. In Claim \ref{c1}, we showed that to every $x \in G$ we may associate a set of planes $S_{x} \subset S$ of measure $\gamma_{d,n}(S_{x}) \gtrsim1/\sqrt{H}$ such that $\mathcal{M}f_{V}(\pi_{V}(x)) \gtrsim \sqrt{H}$ holds for all $V \in S_{x}$. Writing $G_{V} := \{x \in G : V \in S_{x}\}$ for $V \in S$, it follows that
\begin{displaymath} \int_{S} \calH^{n}(G_{V}) \, d\gamma_{d,n}(V) = \int_{G} \gamma_{d,n}(S_{x}) \, d\calH^{n}(x) \gtrsim \frac{1}{\sqrt{H}}. \end{displaymath}
Recalling from \eqref{form73} that $\calH^{n}(G_{V}) \leq \calH^{n}(G) \lesssim 1$ for all $V \in S$, we infer that the subset
\begin{displaymath} S_{G} := \{V \in S : \calH^{n}(G_{V}) \gtrsim 1/\sqrt{H}\} \end{displaymath}
has measure $\gamma_{d,n}(S_{G}) \gtrsim 1/\sqrt{H}$. The plan is now to verify that the hypotheses of Proposition \ref{prop2} are valid for the subset $S_{G} \subset S$, and with parameter $N \sim \sqrt{H}$ (this "$N$" has nothing to do with $N = KM$). Consider $V \in S_{G}$. By definition, $\calH^{n}(G_{V}) \gtrsim 1/\sqrt{H}$, and
\begin{equation}\label{form74} \mathcal{M}f_{V}(\pi_{V}(x)) \gtrsim \sqrt{H} =: H', \qquad x \in G_{V}. \end{equation}
Write $H_{V} := \pi_{V}(G_{V})$. Then, \eqref{form74} is equivalent to 
\begin{equation}\label{form77} H_{V} \subset \{\mathcal{M}f_{V} \gtrsim H'\}. \end{equation} 
Moreover, recalling that $G_{V} \subset G \subset E_{D}$ is covered by the discs $D_{Q}$, $Q \in \mathcal{L}_{G}$, and using the inequality (based on $D_{Q} \subset B_{Q}$ and the bilipschitz property of $\pi_{V}|_{D_{Q}} \colon D_{Q} \to V$)
\begin{displaymath} \calH^{n}(G_{V} \cap D_{Q}) \sim \calH^{n}(\pi_{V}(G_{V} \cap D_{Q})) \leq \calH^{n}(\pi_{V}(G_{V}) \cap \pi_{V}(B_{Q})), \qquad Q \in \mathcal{L}_{G}, \, V \in S, \end{displaymath}
we find that
\begin{align} \int_{\{\mathcal{M}f_{V} \gtrsim H'\}} f_{V}(t) \, d\calH^{n}(t) & \geq \int_{H_{V}} f_{V}(t) \, d\calH^{n}(t) \notag\\
& = \sum_{Q \in \mathcal{L}_{G}} \calH^{n}(\pi_{V}(G_{V}) \cap \pi_{V}(B_{Q})) \notag\\
& \gtrsim \sum_{Q \in \mathcal{L}_{G}} \calH^{n}(G_{V} \cap D_{Q}) = \calH^{n}(G_{V}) \notag \\
&\label{form76} \gtrsim 1/\sqrt{H} = 1/H', \qquad V \in S_{G}. \end{align}
Now, \eqref{form76} says that the hypothesis \eqref{form19} of Proposition \ref{prop2} is satisfied for the set of leaves $\mathcal{G} := \mathcal{L}_{G}$, the set of planes $S_{G} \subset S$, and with the constant "$H'$" in place of "$N$". Moreover, by their definition below \eqref{form94}, all the cubes $Q \in \mathcal{L}_{G}$ satisfy the PBP condition with common plane $V_{0}$:
\begin{displaymath} \calH^{n}(\pi_{V}(E \cap B_{Q})) \geq \delta \mu(Q), \qquad Q \in \mathcal{L}_{Q}, \, V \in S = B_{G(d,n)}(V_{0},\tfrac{\delta}{2}). \end{displaymath} 
Consequently, Proposition \ref{prop2} states that if the parameter $H'$ is chosen large enough, depending only on $C_{0}$ and $\delta$, then 
\begin{equation}\label{form110} \w(\calT) \gtrsim c\delta (H')^{-1} \cdot \gamma_{d,n}(S_{G}) \sim 1/H. \end{equation}
As explained above \eqref{form73}, choosing $H' = \sqrt{H}$ this big means forces us to choose the parameter $M \geq 1$ large enough in a manner depending on 
\begin{displaymath} \alpha \sim_{d,\delta} 1, \, C_{0}, \, H \sim_{C_{0},\delta} 1, \, d, \, \delta,\, \epsilon, \, \theta \sim_{C_{0},d,\delta} 1. \end{displaymath} 
So, in fact $M \sim_{C_{0}d,\delta,\epsilon} 1$, as claimed in Proposition \ref{treeProp}. Since the lower bound for $\w(\calT)$ in \eqref{form110} only depends on the $n$-regularity and PBP constant of $E$, the proof of Proposition \ref{treeProp} is complete. 

\medskip

Since Proposition \ref{mainProp} follows from Proposition \ref{treeProp}, and the construction of heavy trees in Section \ref{s:trees}, we have now proved Proposition \ref{mainProp}. As we recorded in Lemma \ref{l:WGL}, this implies that $n$-regular sets $E \subset \R^{d}$ having PBP satisfy the WGL, and then the BPLG property follows from Theorem \ref{t:DS}. This completes the proof of Theorem \ref{main}. 

\appendix

\section{A variant of the Lebesgue differentiation theorem}\label{appA}

Here we prove Lemma \ref{lemma1}, which we restate below for the reader's convenience:

\begin{lemma}\label{lemma1a} Fix $M,d,\gamma \geq 1$ and $c > 0$. Then, the following holds if $A = A_{d} \geq 1$ is large enough, depending only on $d$ (as in "$\R^{d}$"), and 
\begin{equation}\label{form41a} N > A^{(\gamma + 1)^{2}}M^{\gamma + 2}/c \end{equation} 
Let $\mathcal{B}$ be a collection of balls contained in $[0,1)^{d} \subset \R^{d}$, and associate to every $B \in \mathcal{B}$ a weight $w_{B} \geq 0$. Set
\begin{displaymath} f = \sum_{B \in \mathcal{B}} w_{B}\mathbf{1}_{B}, \end{displaymath}
and write $H_{N} := \{\mathcal{M}f \geq N\}$, where $\mathcal{M}f$ is the Hardy-Littlewood maximal function of $f$. Assume that
\begin{equation}\label{form10} \int_{H_{N}} f(x) \, dx \geq cN^{-\gamma}, \end{equation}
Then, there exists a collection $\mathcal{R}_{\mathrm{heavy}}$ of disjoint cubes such that the "sub-functions"
\begin{displaymath} f_{R} := \mathop{\sum_{B \in \mathcal{B}}}_{B \subset R} w_{B}\mathbf{1}_{B}, \qquad R \in \mathcal{R}_{\mathrm{heavy}}, \end{displaymath}
satisfy the following properties:
\begin{equation}\label{form14} \sum_{R \in \mathcal{R}_{\mathrm{heavy}}} \|f_{R}\|_{1} \geq c2^{-2(\gamma + 1)}N^{-\gamma} \quad \text{and} \quad \|f_{R}\|_{1} > M|R|, \qquad R \in \mathcal{R}_{\mathrm{heavy}}. \end{equation} 
\end{lemma}

\begin{remark} Comparing with \eqref{form10}, the first property in \eqref{form14} states a non-negligible fraction of the $L^{1}$-mass of $f$ is preserved in the functions $f_{R}$, $R \in \mathcal{R}_{\mathrm{heavy}}$. In conjunction with \eqref{form41a}, the second property in \eqref{form14} states that the functions $f_{R}$ can be arranged to have arbitrarily high $L^{1}$-density in $R$, at the cost of choosing the parameter $N$ large. \end{remark}

\begin{remark} While proving Lemma \ref{lemma1a}, we will apply the well-known inequalities
\begin{equation}\label{form113} \int_{\{\mathcal{M}f > C\lambda\}} |f(x)| \, dx \lesssim \lambda \cdot |\{\mathcal{M}f > \lambda\}| \lesssim \int_{\{f > \lambda/2\}} |f(x)| \, dx, \end{equation}
valid for $f \in L^{1}(\R^{d})$, every $\lambda > 0$, and a certain constant $C = C_{d} \geq 1$. The first inequality in \eqref{form113} is stated in \cite[(6)]{MR247534}, but we provide the short details. Let $C = C_{d} \geq 1$ be a constant to be specified in a moment. Write $\Omega_{h} := \{\mathcal{M}f > h\}$ for $h > 0$. For every $x \in \Omega_{C\lambda}$, choose a radius $r_{x} > 0$ such that, denoting $B_{x} := B(x,r_{x})$, we have
\begin{equation}\label{form114} C\lambda \leq \frac{1}{|B_{x}|} \int_{B_{x}} |f(x)| \, dx \leq 2C\lambda. \end{equation} 
This is possible, since $f \in L^{1}(\R^{d})$. For example, one can take $r_{x} > 0$ to be the supremum of the (non-empty and bounded set of) radii such that the left hand inequality in \eqref{form114} holds. The radii "$r_{x}$" are uniformly bounded, again by $f \in L^{1}(\R^{d})$. We then apply the $5r$-covering lemma to the balls $\tfrac{1}{5}B_{x}$ to obtain a countable sub-sequence $\{B_{i}\}_{i \in \N} \subset \{B_{x}\}_{x \in \Omega_{C\lambda}}$ with the properties that (i) the balls $\tfrac{1}{5}B_{i}$ are disjoint, and (ii) the balls $B_{i}$ cover $\bigcup \{\tfrac{1}{5}B_{x} : x \in \Omega_{C\lambda}\} \supset \Omega_{C\lambda}$. We observe that if $C = C_{d} \geq 1$ is large enough, it follows from \eqref{form114} that $\tfrac{1}{5}B_{i} \subset \Omega_{\lambda}$ for all $i \in \N$. Consequently,
\begin{align*} |\Omega_{\lambda}| & \stackrel{\textup{(i)}}{\geq} \sum_{i \in \N} |\tfrac{1}{5}B_{i}| \sim \sum_{i \in \N} |B_{i}| \stackrel{\eqref{form114}}{\geq} \frac{1}{2C\lambda} \sum_{i \in \N} \int_{B_{i}} |f(x)|  \, dx \stackrel{\textup{(ii)}}{\geq} \frac{1}{2C\lambda} \int_{\Omega_{C\lambda}} |f(x)| \, dx, \end{align*} 
as desired. For the second inequality in \eqref{form113}, see \cite[(5), p. 7]{MR0290095}. \end{remark}

\begin{proof}[Proof of Lemma \ref{lemma1a}] We begin with an initial reduction. If $f \notin L^{1}([0,1)^{d})$, there is nothing to prove: then $\mathcal{R}_{\mathrm{heavy}} := \{[0,1)^{d}\}$ satisfies the conclusions \eqref{form14}. So, assume that $f \in L^{1}([0,1)^{d})$, and hence $f \in L^{1}(\R^{d})$, since $\spt f \subset [0,1)^{d}$. Let $C = C_{d} \geq 1$ be the constant from \eqref{form113}. Choosing $N/(2C) < \lambda < N/C$, and combining the inequalities \eqref{form113} with the main assumption \eqref{form10}, we find that
\begin{displaymath} \int_{\{f \geq N/(2C)\}} f(x) \, dx \gtrsim \int_{H_{N}} f(x) \, dx \geq cN^{-\gamma}. \end{displaymath} 
With this in mind, we replace $N$ by $N/(2C)$, and we re-define $H_{N}$ to be the set $H_{N} := \{x : f(x) \geq N\}$. As we just argued, the hypothesis \eqref{form10} remains valid with the new notation, possibly with slightly worse constants.

Fix $N \geq 1$ and abbreviate
\begin{displaymath} \theta := cN^{-\gamma} > 0. \end{displaymath}
It would be helpful if the elements in $\mathcal{B}$ were dyadic cubes instead of arbitrary balls, so we first perform some trickery to reduce (essentially) to this situation. There exist $d + 1$ dyadic systems $\mathcal{D}_{1},\mathcal{D}_{2},\ldots,\mathcal{D}_{d + 1}$ with the following property: every cube $Q \subset [0,1)^{d}$, and consequently every ball $B \subset [0,1)^{d}$, is contained in a dyadic cube $R \in \mathcal{D}_{1} \cup \ldots \cup \mathcal{D}_{d + 1}$ with $|R| \leq C_{d}|Q|$ (resp. $|R| \leq C_{d}|B|$). The constant "$d + 1$" is not crucial -- any dimensional constant would do. The fact that $d + 1$ systems in $\R^{d}$ suffice was shown by Mei \cite{MR1993970}, but such "adjacent" dyadic systems can even be produced in metric spaces, see \cite{MR2901199}. 

In particular, for every $B \in \mathcal{B}$, we may assign an index $i = i_{B} \in \{1,\ldots,d + 1\}$, possibly in a non-unique way, such that $B \subset Q'$ for some $Q' \in \mathcal{D}_{i}$ with $|Q'| \leq C_{d}|B|$. We let $\mathcal{B}_{i}$ be the set of balls in $\mathcal{B}$ with fixed index $i \in \{1,\ldots,d + 1\}$, and we write
\begin{displaymath} f_{i} := \sum_{B \in \mathcal{B}_{i}} w_{B}\mathbf{1}_{B}, \qquad i \in \{1,\ldots,d + 1\}. \end{displaymath}
We claim that there exists $i \in \{1,\ldots,d + 1\}$ such that if $H_{N/(d + 1)}^{i} := \{x : f_{i}(x) \geq N/(d + 1)\}$, then
\begin{equation}\label{form9} \int_{H_{N/(d + 1)}^{i}} f_{i}(x) \, dx \geq \frac{\theta}{(d + 1)^{2}}. \end{equation}
Indeed, one notes that if $x \in H_{N}$ is fixed, then $f_{1}(x) + \ldots + f_{d + 1}(x) = f(x) \geq N$, and hence there exists $i = i_{x} \in \{1,\ldots,d + 1\}$ such that $f_{i}(x) \geq f(x)/(d + 1) \geq N/(d + 1)$. In particular $x \in H_{N/(d + 1)}^{i}$. Then $\mathbf{1}_{H_{N/(d + 1)}^{i}}(x)f_{i}(x) \geq f(x)/(d + 1)$ for this particular $i$, and
\begin{displaymath} \sum_{i = 1}^{d + 1} \int_{H_{N}^{i}(x)} f_{i}(x) \, dx \geq \int_{H_{N}} \sum_{i = 1}^{d + 1} \mathbf{1}_{H_{N/(d + 1)}^{i}}(x)f_{i}(x) \, dx \geq \frac{1}{d + 1}\int_{H_{N}} f(x) \, dx \geq \frac{\theta}{d + 1}.  \end{displaymath}
This implies \eqref{form9}. We now fix $i \in \{1,\ldots,d + 1\}$ satisfying \eqref{form9}. Then $f_{i}$ satisfies the hypothesis \eqref{form10} with the slightly worse constants "$\theta/(d + 1)^{2}$" and "$N/(d + 1)$". Also, it evidently suffices to prove the claimed lower bounds in \eqref{form14} for "$f_{i}$" and its "sub-functions"
\begin{displaymath} f_{R}^{i} := \mathop{\sum_{B \in \mathcal{B}_{i}}}_{B \subset R} w_{B}\mathbf{1}_{B} \leq f_{R} \end{displaymath}
in place of $f$ and the "sub-functions" $f_{R}$. Let us summarise the findings: by passing from $\mathcal{B}$ to $\mathcal{B}_{i}$ and from $f$ to $f_{i}$ if necessary, we may assume that every ball in the original collection "$\mathcal{B}$" is contained in an element "$R$" of some dyadic system "$\calD$" with $|R| \leq C_{d}|B|$. We make this \emph{a priori} assumption in the sequel. 

For every dyadic cube $R \in \mathcal{D}$, we define the weight
\begin{displaymath} \ww_{R} := \mathop{\sum_{B \in \mathcal{B}}}_{B \sim R} w_{B}. \end{displaymath}
Here the relation $B \sim R$ means that $B \subset R$, and $|R| \leq C_{d}|B|$. By the previous arrangements, for every $B \in \mathcal{B}$ there exist $\sim_{d} 1$ dyadic cubes $R \in \mathcal{D}$ such that $B \sim R$. It is worth pointing out that
\begin{displaymath} f(x) = \sum_{B \in \mathcal{B}} w_{B}\mathbf{1}_{B}(x) \leq \sum_{R \in \calD} \ww_{R}\mathbf{1}_{R}(x), \qquad x \in [0,1)^{d}, \end{displaymath} 
because if $x \in B \in \mathcal{B}$, then $B \sim R$ for some $R \in \calD$. It follows that $x \in R$, and $w_{B}$ is one of the terms in the sum defining $\ww_{R}$. 

We now begin the proof in earnest. If $\|f\|_{1} > M$ there is nothing to prove: then we simply declare $\mathcal{R}_{\mathrm{heavy}} := \{[0,1)^{d}\}$, and \eqref{form14} is satisfied. So, we may assume that
\begin{equation}\label{form26} \|f\|_{1} \leq M. \end{equation}
We will next perform $k \in \N$ successive stopping time constructions, for some $1 \leq k \leq \gamma + 1$, which will generate a families $\mathcal{R}_{1},\mathcal{R}_{2},\ldots,\mathcal{R}_{k} \subset \mathcal{D}$ of disjoint dyadic cubes. The cubes in $\mathcal{R}_{k + 1}$ will be contained in the union of the cubes in $\mathcal{R}_{k}$. A subset of one of these families will turn out to be the family "$\mathcal{R}_{\mathrm{heavy}}$" whose existence is claimed. 

Let $\mathcal{R}_{1} \subset \mathcal{D}$ be the maximal (hence disjoint) dyadic cubes with the property
\begin{equation}\label{form17} \mathop{\sum_{R' \in \mathcal{D}}}_{R' \supset R} \ww_{R'}\mathbf{1}_{R'}(x) \geq N_{1} := \floor{N/2}, \qquad x \in R. \end{equation} 
Note that the definition is well posed, since the sum on the left hand side of \eqref{form17} is constant on $R$. We first record the easy observation
\begin{equation}\label{form16} H_{N} \subset \bigcup_{R \in \mathcal{R}_{1}} R. \end{equation}
Indeed, if $x \in H_{N}$, then 
\begin{displaymath} \sum_{B \in \mathcal{B}} w_{B}\mathbf{1}_{B}(x) = f(x) \geq N. \end{displaymath}
It then follows from the definition of the coefficients $\ww_{R}$ (and the fact that every $B \in \mathcal{B}$ is contained in some $R \in \mathcal{D}$) that there exist dyadic cubes $R \in \mathcal{D}$ containing $x$ such that \eqref{form17} holds, and in particular $x \in R$ for some $R \in \mathcal{R}_{1}$.

Next, we calculate that
\begin{align} \sum_{R \in \mathcal{R}_{1}} |R| & \leq \sum_{R \in \mathcal{R}_{1}} \frac{1}{N_{1}} \int_{R} \mathop{\sum_{R' \in \mathcal{D}}}_{R' \supset R} \ww_{R'}\mathbf{1}_{R'}(x) \, dx \notag\\
&\label{form32} \leq \frac{1}{N_{1}} \sum_{R' \in \mathcal{D}} \ww_{R'} \mathop{\sum_{R \in \mathcal{R}_{1}}}_{R \subset R'} |R| \leq \frac{1}{N_{1}} \sum_{R' \in \mathcal{D}} \ww_{R'}|R'|,  \end{align} 
since the cubes in $\mathcal{R}_{1}$ are disjoint. Moreover, by \eqref{form26},
\begin{displaymath} \sum_{R' \in \mathcal{D}} \ww_{R'}|R'| \lesssim_{d} \sum_{R' \in \mathcal{D}} \mathop{\sum_{B \in \mathcal{B}}}_{B \sim R'} w_{B}|B| = \sum_{B \in \mathcal{B}} w_{B}|B| \card\{R' : B \sim R'\} \lesssim_{d} \|f\|_{1} \leq M, \end{displaymath}
so
\begin{equation}\label{form11} \sum_{R \in \mathcal{R}_{1}} |R| \leq \frac{AM}{N_{1}} \end{equation}
for some constant $A = A_{d} \geq 1$. The precise relation between this "$A$" and the dimensional constant appearing in the main assumption \eqref{form41a} is that, in the end, we will need $N > (2A)^{\gamma + 1}3^{(\gamma + 1)^{2}}M^{\gamma + 2}/c$. Next, we claim that if $x \in R \in \mathcal{R}_{1}$, then 
\begin{equation}\label{form12} \mathop{\sum_{B \in \mathcal{B}}}_{B \not\subset R} w_{B}\mathbf{1}_{B}(x) \leq \mathop{\sum_{R' \in \mathcal{D}}}_{R' \supsetneq R} \ww_{R'}\mathbf{1}_{R'}(x) < N_{1} \leq N/2. \end{equation} 
The second inequality follows directly from the definition of the maximal cubes $R \in \mathcal{R}_{1}$. Regarding the first inequality, note that if $B \in \mathcal{B}$ is a ball satisfying $x \in B \cap R$ and $B \not\subset R$, then $B \subset R'$ for some strict ancestor $R' \in \mathcal{D}$ of $R$. Then the coefficient $w_{B}$ appears in the sum defining $\ww_{R'}$ for this ancestor $R' \supsetneq R$. As a corollary of \eqref{form12}, and recalling that $f(x) \geq N$ for all $x \in H_{N}$, we record that
\begin{equation}\label{form15} f_{R}(x) := \mathop{\sum_{B \in \mathcal{B}}}_{B \subset R} w_{B}\mathbf{1}_{B}(x) = f(x) - \mathop{\sum_{B \in \mathcal{B}}}_{B \not\subset R} w_{B}\mathbf{1}_{B}(x) \geq \tfrac{1}{2}f(x), \quad x \in R \cap H_{N}, \, R \in \mathcal{R}_{1}. \end{equation}
The proof now splits into two cases: in the first one, we are actually done, and in the second one, a new stopping family $\mathcal{R}_{2}$ will be generated. The case distinction is based on examining the following "heavy" cubes in $\mathcal{R}_{1}$:
\begin{displaymath} \mathcal{R}_{1,\mathrm{heavy}} := \left\{R \in \mathcal{R}_{1} : \|f_{R}\|_{1} > M|R|\right\}. \end{displaymath} 
\subsubsection*{Case 1} Assume first that
\begin{equation}\label{form27} \sum_{R \in \mathcal{R}_{1,\mathrm{heavy}}} \int_{R \cap H_{N}} f(x) \, dx \geq \frac{\theta}{2}. \end{equation} 
Then
\begin{displaymath} \sum_{R \in \mathcal{R}_{1,\mathrm{heavy}}} \|f_{R}\|_{1} \stackrel{\eqref{form15}}{\geq} \frac{1}{2} \sum_{R \in \mathcal{R}_{1,\mathrm{heavy}}} \int_{R \cap H_{N}} f(x) \, dx \geq \frac{\theta}{4}. \end{displaymath}
In this case, we set $\mathcal{R}_{\mathrm{heavy}} := \mathcal{R}_{1,\mathrm{heavy}}$, and the proof terminates, because \eqref{form14} is satisfied.
\subsubsection*{Case 2} Assume next that \eqref{form27} fails, and recall from \eqref{form16} that $H_{N}$ is contained in the union of the cubes in $\mathcal{R}_{1}$. Therefore,
\begin{equation}\label{form28} \sum_{R \in \mathcal{R}_{1,\mathrm{light}}} \int_{R \cap H_{N}} f(x) \, dx \geq \int_{H_{N}} f(x) \, dx - \frac{\theta}{2} \geq \frac{\theta}{2}, \end{equation}
where $\mathcal{R}_{1,\mathrm{light}} = \mathcal{R}_{1} \, \setminus \, \mathcal{R}_{1,\mathrm{heavy}}$.

We now proceed to define the next generation stopping cubes $\mathcal{R}_{2}$. Fix $R_{0} \in \mathcal{R}_{1,\mathrm{light}}$, and consider the maximal dyadic sub-cubes $R \subset R_{0}$ with the property
\begin{equation}\label{form29} \sum_{R \subset R' \subset R_{0}} \ww_{R'}\mathbf{1}_{R'}(x) \geq N_{2} := \floor{N/4}, \qquad x \in R, \end{equation}
Again, the left hand side of \eqref{form29} is constant on $R$, so the stopping condition is well-posed. The cubes so obtained are denoted $\mathcal{R}_{2}(R_{0})$, and we set
\begin{equation}\label{form35} \mathcal{R}_{2} := \bigcup_{R_{0} \in \mathcal{R}_{1,\mathrm{light}}} \mathcal{R}_{2}(R_{0}). \end{equation}
We claim that the (fairly large) part of $H_{N}$ covered by cubes in $\mathcal{R}_{1,\mathrm{light}}$ is remains covered by the cubes in $\mathcal{R}_{2}$. Indeed, fix $x \in R_{0} \cap H_{N}$, where $R_{0} \in \mathcal{R}_{1,\mathrm{light}} \subset \mathcal{R}_{1}$. Then
\begin{displaymath} \mathop{\sum_{R' \in \mathcal{D}}}_{R' \supsetneq R_{0}} \ww_{R'}\mathbf{1}_{R'}(x) < N_{1} \leq N/2 \end{displaymath}
by definitions of $\mathcal{R}_{1}$ and $N_{1}$, so
\begin{displaymath} \mathop{\sum_{R' \in \mathcal{D}}}_{R' \subset R_{0}} \ww_{R'}\mathbf{1}_{R'}(x) \geq N/2, \end{displaymath}
and hence $x$ is contained in some (maximal) dyadic cube $R \subset R_{0}$ satisfying \eqref{form29}.

Arguing as in \eqref{form12}, we infer the following: if $x \in R \in \mathcal{R}_{2}$, then
\begin{equation}\label{form30} \mathop{\sum_{B \in \mathcal{B}}}_{B \not\subset R} w_{B}\mathbf{1}_{B}(x) \leq \mathop{\sum_{R' \in \mathcal{D}}}_{R' \supsetneq R} \ww_{R'}\mathbf{1}_{R'}(x) < N_{1} + N_{2} \leq \frac{3N}{4}. \end{equation} 
Indeed, the first inequality follows exactly as in \eqref{form12}. To see the second inequality, split the cubes $R' \supsetneq R$ into the ranges $R \subsetneq R' \subset R_{0}$ and $R_{0} \subsetneq R \subset [0,1)^{d}$, where $R_{0} \in \mathcal{R}_{1}$. Then, use the definitions of the stopping cubes $\mathcal{R}_{1}$ and $\mathcal{R}_{2}$. As a corollary of \eqref{form30}, we infer an analogue of \eqref{form15} for $R \in \mathcal{R}_{2}$:
\begin{equation}\label{form31} f_{R}(x) = \mathop{\sum_{B \in \mathcal{B}}}_{B \subset R} w_{B}\mathbf{1}_{B}(x) = f(x) - \mathop{\sum_{B \in \mathcal{B}}}_{B \not\subset R} w_{B}\mathbf{1}_{B}(x) \geq \tfrac{1}{4}f(x), \quad x \in R \cap H_{N}, \, R \in \mathcal{R}_{2}. \end{equation}
We next estimate the total volume of the cubes in $\mathcal{R}_{2}$. Fix $R_{0} \in \mathcal{R}_{1,\mathrm{light}}$, and first estimate
\begin{align*} \mathop{\sum_{R \in \mathcal{R}_{2}}}_{R \subset R_{0}} |R| \leq \mathop{\sum_{R \in \mathcal{R}_{2}}}_{R \subset R_{0}} \frac{1}{N_{2}} \int_{R} \sum_{R \subset R' \subset R_{0}} \ww_{R'}\mathbf{1}_{R'}(x) \, dx \leq \frac{1}{N_{2}} \mathop{\sum_{R' \in \mathcal{D}}}_{R' \subset R_{0}} \ww_{R'}|R'|. \end{align*} 
Of course, this computation was just a repetition of \eqref{form32}. Also the next estimate can be carried out in the same way as the estimate just below \eqref{form32}:
\begin{displaymath} \mathop{\sum_{R' \in \mathcal{D}}}_{R' \subset R_{0}} \ww_{R'}|R'| \leq A \|f_{R_{0}}\|_{1} \leq AM|R_{0}|, \qquad R_{0} \in \mathcal{R}_{1,\mathrm{light}}. \end{displaymath} 
Combining the previous two displays, the stopping cubes in $\mathcal{R}_{2}(R_{0})$ have total volume $\leq AM|R_{0}|/N_{2}$ for every $R_{0} \in \mathcal{R}_{1,\mathrm{light}}$. Therefore,
\begin{equation}\label{form33} \sum_{R \in \mathcal{R}_{2}} |R| = \sum_{R_{0} \in \mathcal{R}_{1,\mathrm{light}}} \sum_{R \in \mathcal{R}_{2}(R_{0})} |R| \leq \frac{AM}{N_{2}} \sum_{J_{0} \in \mathcal{J}_{1,\mathrm{light}}} |R_{0}| \leq \frac{A^{2}M^{2}}{N_{1}N_{2}}, \end{equation}
recalling \eqref{form11}. Since $N \gg \max\{A,M\}$, this means that the total volume of the stopping cubes tends to zero rapidly as their generation increases. 

We are now prepared to make another case distinction, this time based on the heavy sub-cubes in $\mathcal{R}_{2}$:
\begin{displaymath} \mathcal{R}_{2,\mathrm{heavy}} := \left\{R \in \mathcal{R}_{2} : \|f_{R}\|_{1} > M|R| \right\}. \end{displaymath}
\subsubsection*{Case 2.1} Assume first that
\begin{equation}\label{form34} \sum_{R \in \mathcal{R}_{2,\mathrm{heavy}}} \int_{R \cap H_{N}} f(x) \, dx \geq \frac{\theta}{4}. \end{equation}
Then,
\begin{equation}\label{form37}  \sum_{R \in \mathcal{R}_{2,\mathrm{heavy}}} \|f_{R}\|_{1} \stackrel{\eqref{form31}}{\geq} \frac{1}{4} \sum_{R \in \mathcal{R}_{2,\mathrm{heavy}}} \int_{R \cap H_{N}} f(x) \, dx \geq \frac{\theta}{16}. \end{equation}
In this case, we declare $\mathcal{R}_{\mathrm{heavy}} := \mathcal{R}_{2,\mathrm{heavy}}$, and we see that \eqref{form14} is satisfied.
\subsubsection*{Case 2.2} Assume then that \eqref{form34} fails. Since the part of $H_{N}$ contained in the $\mathcal{R}_{1,\mathrm{light}}$-cubes is also contained in the $\mathcal{R}_{2}$-cubes (as established right below \eqref{form35}), we deduce from \eqref{form28} that
\begin{displaymath} \sum_{R \in \mathcal{R}_{2,\mathrm{light}}} \int_{R \cap H_{N}} f(x) \, dx \geq \sum_{R \in \mathcal{R}_{1,\mathrm{light}}} \int_{R \cap H_{N}} f(x) \, dx - \frac{\theta}{4} \geq \frac{\theta}{4}. \end{displaymath} 
Here of course $\mathcal{R}_{2,\mathrm{light}} := \mathcal{R}_{2} \, \setminus \, \mathcal{R}_{2,\mathrm{heavy}}$. So, we find ourselves in a situation analogous to \eqref{form28}, except that the integral of $f\mathbf{1}_{H_{N}}$ over the light cubes has decreased by half. 

Repeating the construction above, we proceed to define -- inductively -- new collections of stopping cubes. The stopping cubes $\mathcal{R}_{k}$ are contained in the the union of the stopping cubes $\mathcal{R}_{k - 1,\mathrm{light}}$, and they are defined as the maximal sub-cubes "$R$" of $R_{0} \in \mathcal{R}_{k - 1,\mathrm{light}}$ satisfying 
\begin{displaymath} \sum_{R \subset R' \subset R_{0}} \ww_{R'}\mathbf{1}_{R'}(x) \geq N_{k} := \floor{N/2^{k}}, \qquad x \in R. \end{displaymath}
Repeating the argument under \eqref{form35}, this definition ensures that the part of $H_{N}$ covered by the cubes in $\mathcal{R}_{k - 1,\mathrm{light}}$ remains covered by the union of the cubes in $\mathcal{R}_{k}$. Moreover, induction shows that
\begin{equation}\label{form39} \sum_{R \in \mathcal{R}_{k - 1,\mathrm{light}}} \int_{R \cap H_{N}} f(x) \, dx \geq 2^{-k + 1}\theta, \qquad k \geq 1. \end{equation} 
The general analogue of the inequality \eqref{form31} is
\begin{equation}\label{form38} f_{R}(x) \geq 2^{-k}f(x), \qquad x \in R \cap H_{N}, \, R \in \mathcal{R}_{k}, \end{equation}
and the total volume of the cubes in $\mathcal{R}_{k}$ satisfies
\begin{equation}\label{form36} \sum_{R \in \mathcal{R}_{k}} |R| \leq \frac{A^{k}M^{k}}{N_{1}\cdots N_{k}}, \end{equation}
in analogy with \eqref{form33}. Once the cubes in $\mathcal{R}_{k}$ have been constructed, we split into two cases, depending on whether
\begin{equation}\label{form40} \sum_{R \in \mathcal{R}_{k,\mathrm{heavy}}} \int_{R \cap H_{N}} f(x) \, dx \geq 2^{-k}\theta \quad \text{or} \quad \sum_{R \in \mathcal{R}_{k,\mathrm{light}}} \int_{R \cap H_{N}} f(x) \, dx \geq 2^{-k}\theta. \end{equation} 
One of these cases must occur because of \eqref{form39}, and the covering property stated above \eqref{form39}. In the first case, \eqref{form38} shows that
\begin{displaymath} \sum_{R \in \mathcal{R}_{k,\mathrm{heavy}}} \|f_{R}\|_{1} \geq 2^{-k} \sum_{R \in \mathcal{R}_{k,\mathrm{heavy}}} \int_{R \cap H_{N}} f(x) \, dx \geq 2^{-2k}\theta, \end{displaymath} 
and the proof of \eqref{form14} concludes if $k \leq \gamma + 1$. So, the only remaining task is to show that the first case \textbf{must} occur for some $k \leq \gamma + 1$. Indeed, if the second case of \eqref{form40} occurs for any $k \geq 1$, we have
\begin{displaymath} c2^{-k}N^{-\gamma} = 2^{-k}\theta \leq \sum_{R \in \mathcal{R}_{k,\mathrm{light}}} \|f_{R}\|_{1} \leq M \sum_{R \in \mathcal{R}_{k}} |R| \stackrel{\eqref{form36}}{\leq} \frac{A^{k}M^{k + 1}}{N_{1}\cdots N_{k}}.  \end{displaymath}
Recalling that $N_{k} = \floor{N/2^{k}} \geq N/3^{k}$, hence $N_{1}\cdots N_{k} \geq N^{k}3^{-k^{2}}$, this yields
\begin{displaymath} N^{k - \gamma} \leq \frac{3^{k^{2}}(2A)^{k}M^{k + 1}}{c}. \end{displaymath}
Assuming that $N > 3^{(\gamma + 1)^{2}}(2A)^{\gamma + 1}M^{\gamma + 2}/c$ (in agreement with \eqref{form41a}), the inequality above cannot hold for $k = \gamma + 1$. Thus, the "heavy" case of \eqref{form40} occurs latest at step $k = \gamma + 1$. The proof of the lemma is complete. \end{proof}

\bibliographystyle{plain}
\bibliography{references}

\end{document}